\documentclass[11pt]{article}
	\usepackage{ifthen}
	\newboolean{isTRAC}\setboolean{isTRAC}{false}
	\newcommand{\cstmHeader}{\centering}

	\usepackage{_aux/trac2016} 
\usepackage[utf8]{inputenc}           
\usepackage[T1]{fontenc}
\usepackage[ngerman,english]{babel}
\usepackage{textcomp}
\usepackage{xcolor}
\usepackage[markup=nocolor]{changes}
	\definechangesauthor{AC}
\usepackage{soul} 
\usepackage{lineno}
\usepackage{empheq}
\usepackage{enumerate}
\usepackage{caption}
\usepackage{subcaption}	
\usepackage{graphicx}
\usepackage{rotating}
\usepackage{pdfpages}
\usepackage{amssymb}
\usepackage{amsfonts}
\usepackage{amsmath}
\usepackage{mathtools}
\usepackage{thmtools}
\usepackage{units}
\usepackage{siunitx}
\usepackage{xfrac}
\usepackage{cancel}
\usepackage{longtable}
\usepackage{multirow}
\usepackage{makecell}
\usepackage{booktabs}
\usepackage{colortbl}
\usepackage{csquotes}
\usepackage[backend=biber,style=numeric,citestyle=ieee,sorting=none,sortcites=true]{biblatex}
\usepackage{listings}
\usepackage{./_aux/mcode}
\usepackage{algorithm, algpseudocode} 

\usepackage{hyperref}
\usepackage[noabbrev,capitalise]{cleveref} 
\usepackage{tikz}
\usetikzlibrary{decorations.pathreplacing}
\usetikzlibrary{shapes,arrows,patterns}
\usepackage{pgfplots}
\pgfplotsset{compat=newest} 
\pgfplotsset{plot coordinates/math parser=false}
\pgfplotsset{yticklabel style={text width=2.5em,align=right}}
\pgfplotsset{/pgf/number format/.cd, fixed,precision=6}
\newlength\myheight 
\newlength\mywidth  
\usepackage{pgf-pie}
\definecolor{tumblue}{HTML}{0065BD}
\definecolor{tumblue1}{HTML}{98C6EA}
\definecolor{tumblue2}{HTML}{64A0C8}
\definecolor{tumblue3}{HTML}{0073CF}
\definecolor{tumblue4}{HTML}{005293}
\definecolor{tumblue5}{HTML}{003359}
\definecolor{tumgreen}{HTML}{A2AD00}
\definecolor{tumorange}{HTML}{E37222}
\definecolor{tumivory}{HTML}{DAD7CB}
\definecolor{tumviolet}{HTML}{69085A}
\definecolor{tumred}{HTML}{C4071B}
\definecolor{bg-gray}{HTML}{E5E5E5}
\definecolor{bg-red}{HTML}{FFDDDD}
\definecolor{bg-green}{HTML}{DDFFDD}
\definecolor{bg-blue}{HTML}{DDDDFF}
\definecolor{text-gray}{HTML}{404040}
\definecolor{text-red}{HTML}{AA0000}
\definecolor{text-green}{HTML}{00AA00}
\definecolor{text-blue}{HTML}{0000AA}
\hypersetup{pdftitle = {A New Framework for H2-Optimal Model Reduction},
  			pdfsubject = {},
  			pdfauthor = {Alessandro Castagnotto, Boris Lohmann},
  			pdfkeywords = {MOR, Model Order Reduction, Krylov Subspace Methods, H2 Pseudo-Optimality, Rational Interpolation},
  			pdfcreator = {},
  			pdfproducer = {},
            pdffitwindow = false,
            linktocpage = true,		
            colorlinks = true,		
            linkcolor = tumblue,	
            citecolor = tumblue,	
            urlcolor = tumblue,		
            }
\DeclareUrlCommand\doi{\def\UrlLeft##1\UrlRight{DOI:\nobreakspace\href{http://dx.doi.org/##1}{##1}}\urlstyle{rm}}
\crefformat{equation}{{#2{\color{black}(}#1{\color{black})}#3}}
\Crefformat{equation}{{#2{\color{black}(}#1{\color{black})}#3}}
\crefformat{enumeratei}{{#2{\color{black}(}#1{\color{black})}#3}}
\Crefformat{enumeratei}{{#2{\color{black}(}#1{\color{black})}#3}}
\crefname{enumeratei}{}{}
\newcommand{\ts}{\!} 
\newcommand{\teq}{\ts=\ts}
\newcommand{\tin}{\ts\in\ts}
\newcommand{\tapprox}{\ts\approx\ts}

\renewcommand{\phi}{\varphi}           
\newcommand{\fo}{N} 							
\newcommand{\ro}{n} 							
\newcommand{\is}{m} 							
\newcommand{\os}{p} 							

\newcommand{\iu}{{i\mkern1mu}}

\newcommand{\defeq}{\vcentcolon=}

\newcommand{\Hinf}{\mathcal{H}_\infty}                                       
\newcommand{\Htwo}{\texorpdfstring{\mathcal{H}_2}{H2}}                       
\newcommand{\HtwoText}{\texorpdfstring{$\mathcal{H}_2$}{H2}}                 
\newcommand{\norm}[2]{\lVert #1 \rVert_{#2}}                                 
\newcommand{\normHtwo}[1]{\lVert #1 \rVert_{\Htwo}}                          
\newcommand{\Reals}{\mathbb{R}}
\newcommand{\Complex}{\mathbb{C}}
\renewcommand{\Im}{\mathrm{Im}}
\DeclareMathOperator{\Image}{\mathcal{R}}

\DeclareMathOperator{\trace}{tr}                     
\DeclareMathOperator{\diag}{diag}                    
\newcommand{\trans}[1]{#1^\top}
\newcommand{\htrans}[1]{#1^H}

\newcommand{\inv}[1]{#1^{-1}}
\newcommand{\invt}[1]{#1^{-\top}}

\newcommand{\ei}{e_i}

\newcommand{\kIrka}{k_{\mathcal{H}_2}}
\newcommand{\kCirka}{k_{\model}}

\newcommand{\shiftSetInput}{\left\{\sigma_i\right\}_{i=1}^\ro}

\newcommand{\tanDirSetInput}{\left\{r_i\right\}_{i=1}^\ro}

\newcommand{\tanDirSetOutputI}{\left\{l_i\right\}_{i=1}^\ro}

\newcommand{\shiftSetInputK}[1]{\left\{\sigma_i^{#1}\right\}_{i=1}^\ro}
\newcommand{\tanDirSetInputK}[1]{\left\{r_i^{#1}\right\}_{i=1}^\ro}
\newcommand{\tanDirSetOutputIK}[1]{\left\{l_i^{#1}\right\}_{i=1}^\ro}

\newcommand{\Asig}{A_{\sigma_i}}
\newcommand{\Cost}[2]{\mathcal{C}_{#2}\left(#1\right)}
\newcommand{\nLU}{n_{LU}}

\newcommand{\inRes}{\hat{b}_i}
\newcommand{\outRes}{\hat{c}_i}

\newcommand{\prim}{P}
\newcommand{\Vprim}{V^\prim}
\newcommand{\Wprim}{W^\prim}
\newcommand{\Sprim}[1]{S^{\prim {#1}}}

\newcommand{\Rprim}{R^\prim}
\newcommand{\Lprim}{L^\prim}

\newcommand{\model}{\mu}
	\newcommand{\SigmaM}[1]{\Sigma_\model^{#1}}
	\newcommand{\GM}{G_\model}
	\newcommand{\GMp}{G'_\model}
	\newcommand{\SigmaMr}[1]{\Sigma_{\model,r}^{#1}}
	\newcommand{\GMr}{G_{\model,r}}
	\newcommand{\GMrp}{G'_{\model,r}}
\newcommand{\nM}[1]{\ro_\model^{#1}}
\newcommand{\VM}[1]{V_\model^{#1}}
\newcommand{\WM}[1]{W_\model^{#1}}
\newcommand{\WMtrans}[1]{\trans{\left(W_\model^{#1}\right)}}
\newcommand{\VMr}[1]{V_{\model,r}^{#1}}
\newcommand{\WMr}[1]{W_{\model,r}^{#1}}
\newcommand{\WMrtrans}[1]{\trans{\left(W_{\model,r}^{#1}\right)}}
\newcommand{\SModel}[1]{S_\model^{#1}}
\newcommand{\RModel}[1]{R_\model^{#1}}
\newcommand{\LModel}[1]{L_\model^{#1}}

\newcommand{\shiftMi}{\sigma_{\model,i}}
\newcommand{\tanDirInputMi}{r_{\model,i}}
\newcommand{\tanDirOutputMi}{l_{\model,i}}
\newcommand{\shiftMj}{\sigma_{\model,j}}
\newcommand{\tanDirInputMj}{r_{\model,j}}
\newcommand{\tanDirOutputMj}{l_{\model,j}}

\newcommand{\tripletMj}{\left(\shiftMj,\tanDirInputMj,\tanDirOutputMj\right)}

\newcommand{\opt}{{\ast}}
\newcommand{\SOpt}[1]{S_\opt^{#1}}
\newcommand{\ROpt}[1]{R_\opt^{#1}}
\newcommand{\LOpt}[1]{L_\opt^{#1}}
\newcommand{\shiftOpti}{\sigma_{\opt,i}}
\newcommand{\tanDirInputOpti}{r_{\opt,i}}
\newcommand{\tanDirOutputOpti}{l_{\opt,i}}

\newcommand{\shiftSetInputOpt}{\left\{\shiftOpti\right\}_{i=1}^\ro}
\newcommand{\tanDirSetInputOpt}{\left\{\tanDirInputOpti\right\}_{i=1}^\ro}
\newcommand{\tanDirSetOutputIOpt}{\left\{\tanDirOutputOpti\right\}_{i=1}^\ro}

\newcommand{\tripletOpti}{\left(\shiftOpti,\tanDirInputOpti,\tanDirOutputOpti\right)}

\newcommand{\error}{\epsilon_{\Htwo}}
\newcommand{\errorEst}{\tilde{\epsilon}_{\Htwo}}

	\graphicspath{{./_pics/}}

\begin{document}
	\thevolume{}
	\date{March 2017}
	
	\title{A New Framework for \HtwoText-Optimal Model Reduction\thanks{The work related to this contribution is supported by the German Research Foundation (DFG), Grant LO408/19-1.}}
	\newcommand{\runtitle}{A New Framework for \HtwoText-Optimal Model Reduction}
	\author{Alessandro Castagnotto\footnote{Chair of Automatic Control, Technical University of Munich, Boltzmannstr. 15, D-85748 Garching} \footnote{Corresponding author: \href{mailto:a.castagnotto@tum.de}{a.castagnotto@tum.de}}\and Boris Lohmann\footnotemark[2]}

	\maketitle

\begin{abstract}
	In this contribution, a new framework for \HtwoText-optimal reduction of multiple-input, multiple-output linear dynamical systems by tangential interpolation is presented. The framework is motivated by the local nature of both tangential interpolation and \HtwoText-optimal approximations. 
	The main advantage is given by a decoupling of the cost of optimization from the cost of reduction, resulting in a significant speedup in \HtwoText-optimal reduction. 
	In addition, a middle-sized surrogate model is produced at no additional cost and can be used e.g. for error estimation.
	Numerical examples illustrate the new framework, showing its effectiveness in producing \HtwoText-optimal reduced models at a far lower cost than conventional algorithms.
	The paper ends with a brief discussion on how the idea behind the framework can be extended to approximate further system classes, thus showing that this truly is a general framework for interpolatory \HtwoText\ reduction rather than just an additional reduction algorithm.
\end{abstract}

%
\section{Introduction}\label{sec:introduction}
	Recent advances in the development of complex technical systems are partly driven by the advent of software tools that allow their computerized modeling and analysis, drastically reducing the resources required during development.
	Using information about the system's geometry, material properties and boundary conditions, a dynamical model can be derived in an almost automated way, allowing for design optimization and verifications using a \emph{virtual prototype}. 
	Depending on the complexity of the model at hand and the required accuracy during investigation, such models quite easily reach a high complexity. 
	As a consequence, simulations and design optimizations based on these models may require excessive computational resources or even become unfeasible.
	In addition, in some applications the models are required to be evaluated in real-time during operations, e.g. in \emph{embedded controllers} or \emph{digital twins} for state observation and predictive maintenance. In this latter case, the computational resources available are particularly limited.
	
	In this context, \emph{model reduction} is an active field of research aimed at finding numerically efficient algorithms to construct low-order, high-fidelity approximations to the high-order models in a numerically efficient way. 
	Amongst all, the most prominent and numerically tractable methods for linear systems include \emph{approximate balanced truncation} \cite{morPen00,morLi00,Saak_2009,benner2014self,morBenKS12,kurschner2016phd,morDruS11,morDruSZ14} and \emph{Krylov subspace methods} or \emph{rational interpolation} \cite{morGri97,morGalVV04,morAntBG10}.
	As these methods generally require only the solution of large sparse systems of equations, they can be applied efficiently even to models of very high order.
	
	In a context where the admissible model order is low, e.g. in real-time applications, it is of particular interest to find the best possible approximation for a given order. This problem has been addressed in terms of different error norms (see \cite{morGlo84,morBenQQ04} for optimal Hankel norm, \cite{helmersson1994model,Kavranoglu_1993,varga2001fast,morFlaBG13,morCasBG17} for optimal $\Hinf$-norm approximations), but only for the case of optimal $\Htwo$-norm approximations there exist algorithms \cite{morGugAB08,morVanGA08,morBeaG09b,morPanJWetal13} that are both numerically tractable and satisfy optimality conditions.
	\HtwoText-optimal reduction methods are based on the repeated reduction of the high-order model until a locally optimal reduced order model satisfying optimality conditions is found. Therefore, if the convergence is slow and the number of iterations is high, the numerical efficiency of the methods is diminished.

	In this contribution, we describe a new reduction framework to increase the numerical efficiency of \HtwoText-optimal reduction methods.
	In this new framework, firstly introduced in \cite{morPan14} and further developed in \cite{morCasPL17} for SISO models, \HtwoText\ optimization of the reduction parameters is performed in a \emph{reduced subspace}, lowering the optimization cost with respect to conventional methods. Through the update of the reduced subspace, optimality of the resulting reduced order model can be established.
	In this paper we generalize the framework, proving its validity for linear time-invariant system with multiple-inputs and multiple-outputs (MIMO) and indicating possible extensions to further system classes. 
	By applying this framework for \HtwoText-optimal reduction, substantial speedup can be achieved. 
	In addition, this framework bears the advantage of producing a \emph{model function}, i.e. a middle-sized surrogate model, at no additional cost. We will give first indications on how to exploit this model function in further applications and outline current research endeavors towards this direction.
	
	The remainder of the paper is structured as follows:
	\cref{sec:preliminaries} briefly revises \HtwoText-optimal reduction for MIMO linear systems, whereas in \cref{sec:costOfH2} we analyze the computational cost tied to these methods. 
	\cref{sec:aNewFramework} presents the main result of this contribution, i.e. a new framework for \HtwoText-optimal reduction.
	\cref{sec:numericalresults} will compare conventional \HtwoText-optimal reduction to the new framework in numerical examples and show the potential for significant speedup in reduction time. 
	In \cref{sec:furtherApplications} we indicate how to apply this framework to \HtwoText-optimal approaches for different system classes, motivating its general nature. Finally, \cref{sec:conclusions} will summarize and conclude the discussion.


\section{Preliminaries}\label{sec:preliminaries}

	\subsection{Model Reduction by Tangential Interpolation} \label{sec:tangentialInterpolation}
		Linear dynamical systems are generally described by state-space models of the form
		\begin{equation}
			\left.
			\begin{aligned}
				E \,\dot{x}(t) &= A\, x(t) + B\, u(t) \\
				y(t) &= C\, x(t) + D\, u(t)
			\end{aligned}
			\quad \right\} \Sigma
			\label{eq:FOM}
		\end{equation}
		where $E\ts\in\ts\Reals^{\fo \ts\times\ts \fo}$ is the regular \emph{descriptor matrix}, $A\ts\in\ts\Reals^{\fo\ts\times\ts \fo}$ is the system matrix and $x\ts\in\ts\Reals^\fo$, $u\ts\in\ts\Reals^\is$, $y\ts\in\ts\Reals^\os$ ($\os,\is\ts\ll\ts \fo$) represent the state, input and output of the system respectively.
		$\Sigma$ denotes the system \eqref{eq:FOM} by its state-space representation.
		The input-output behavior of a linear system \cref{eq:FOM} can be characterized in the frequency domain by $y(s)=G(s)u(s)$, with the rational transfer function matrix
		\begin{equation}\label{eq:G(s)}
			G(s) \defeq C \inv{\left(s E -A\right)} B + D \quad\in \Complex^{\os\times\is},
		\end{equation}
		obtained through Laplace transform of \eqref{eq:FOM} under the assumption $x(t=0)=0$. 
		The construction of a reduced order model (ROM) from the full order model (FOM) \cref{eq:FOM} can be obtained by means of a \emph{Petrov-Galerkin} projection
		\begin{equation}
		\left.
		\begin{aligned}
		\overbrace{\trans{W}E\,V}^{E_r} \, \dot{x}_r(t)\, &= \, \overbrace{\trans{W}A\,V}^{A_r}\, x_r(t) \, +  \, \overbrace{\trans{W} B}^{B_r} \, u(t)\\ 
		y_r(t) \, &= \; \underbrace{C\,V}_{C_r}\, x_r(t) \,+ \; D_r \, u(t)
		\end{aligned}
		\quad \right\} \Sigma_r
		\label{eq:ROM}	
		\end{equation}
		where $x_r\ts\in\ts\Reals^{\ro}$ $(\ro\ts\ll\ts \fo)$ represents the reduced state vector. 
		We will refer to the ROM realization in \eqref{eq:ROM} through $\Sigma_r$ and use the shorthand notation $\Sigma_r = \trans{W} \Sigma V$ to specify the projection matrices used.

		The primary goal of model reduction in the following will be the approximation of the output $y(t)\approx y_r(t)$ for all admissible inputs $u(t)$. This is equivalent to approximating the transfer function $G(s)\approx G_r(s)$.
		To achieve this goal, the appropriate design of projection matrices $V,W$ becomes the primary task of model reduction.
		Note that most commonly, the reduced feed-through matrix $D_r$ is chosen such that $D_r\ts=\ts D$, hence not playing a role in the reduction process. 
		In fact, this is a necessary condition for optimality with respect to the \HtwoText-norm \cite{morGugSW13}. Thus $D$ will be disregarded it in the following. 
		Note hover that some reduction approaches rely on this additional degree of freedom to increase the approximation quality (cp. \cite{morGlo84,morFlaBG13, morCasBG17}).
	
		The design of $V,W$ in the following will be driven by \emph{bitangential Hermite interpolation}, i.e. we are interested in constructing a ROM $\Sigma_r$ whose transfer function $G_r(s)$ satisfies
		\begin{equation}
			\begin{aligned}
			G(\sigma_i) \, r_i &= G_r(\sigma_i) \, r_i\;, 					\qquad &i = 1,\,\dots\,,\,\ro\\
			\trans{l_i}\, G(\sigma_i) &= \trans{l_i}\, G_r(\sigma_j)\;,  			\qquad &i = 1,\,\dots\,,\,\ro\\
			\trans{l_i}\, G'(\sigma_i) r_i &= \trans{l_i}\, G'_r(\sigma_i)\, r_i\;,  \qquad &i = 1,\,\dots\,,\,\ro
			\end{aligned}	
			\label{eq:tangential_interpolation_conditions}
		\end{equation}
		for complex frequencies $\sigma_i \in\Complex $ and input resp. output \emph{tangential directions} $r_i\in\Complex^{\is}$, $l_i\in\Complex^{\os}$.
		The following result indicates how to construct projection matrices $V,W$ to achieve \cref{eq:tangential_interpolation_conditions}.	
		\begin{theorem}[Bitangential Hermite Interpolation \cite{morGalVV04,morBeaG14}]
			\label{thm:tangential_interpolation}
			Consider a full-order model $\Sigma$ as in \eqref{eq:FOM} with transfer function $G(s)$ and let scalar frequencies $\sigma_i\in\Complex$ and vectors $r_i\in \Complex^{\is}$, $l_j\in \Complex^{\os}$ be given such that $\sigma_iE-A$ is nonsingular for $i=1,\,\dots\,,\,\ro$.
			Consider a reduced-order model $\Sigma_r$ as in \eqref{eq:ROM} with transfer function $G_r(s)$, obtained through projection $\Sigma_r=\trans{W}\Sigma V$.
			\begin{enumerate}
				\item If
				\begin{equation}
				(A-\sigma_iE)^{-1}B\,r_i \in\Image(V), \qquad i=1,\,\dots\,,\,\ro \label{eq:input_Krylov}
				\end{equation}
				then $G(\sigma_i) \, r_i = G_r(\sigma_i)\, r_i$.	
				\item If
				\begin{equation}
				\invt{(A-\sigma_iE)}\trans{C}\, l_i \in\Image(W), \qquad i=1,\,\dots\,,\,\ro \label{eq:output_Krylov}
				\end{equation}
				then $\trans{l_i}\cdot G(\sigma_i) = \trans{l_i}\cdot G_r(\sigma_i)$.
				\item If both \cref{eq:input_Krylov} and \cref{eq:output_Krylov} hold, then, in addition,
				\begin{equation}
				\trans{l_i}\, G'(\sigma_i) \, r_i= \trans{l_i}\, G_r'(\sigma_i)\, r_i,
				\qquad i = 1,\,\dots\,,\,\ro,
				\end{equation}
				where $G'(s)$ denotes the first derivative with respect to $s$.
			\end{enumerate}
		\end{theorem}
		

		In addition, note that it is possible to tangentially interpolate higher order derivatives $G^{(k)}(s)$ at frequencies $\sigma_i\in\Complex$ by spanning appropriate Krylov subspaces \cite{morGri97,morGalVV04}.
		In general, Krylov subspaces are defined through a matrix $M\in\Complex^{\fo\times\fo}$, a vector $v\in\Complex^\fo$ and a scalar dimension $\ro\in\mathbb{N}$ as follows:
		\begin{equation}
		\mathcal{K}_\ro(M,v) = \Image{\left(\begin{bmatrix} v& M\,v & M^2\,v & \dots & M^{\ro-1} \end{bmatrix}\right)}.
		\end{equation}
		
		\begin{theorem}[Tangential Moment Matching \cite{morGri97,morGalVV04}]
			\label{thm:moment_matching}
			Consider a full-order model $\Sigma$ as in \eqref{eq:FOM} with transfer function $G(s)$ and let a scalar frequency $\sigma\in\Complex$ and nonzero vectors $r\in \Complex^{\is}$, $l\in \Complex^{\os}$ be given, such that $\sigma E-A$ is nonsingular.
			Consider a reduced-order model $\Sigma_r$ as in \eqref{eq:ROM} with transfer function $G_r(s)$, obtained through the projection $\Sigma_r=\trans{W}\Sigma V$.
			\begin{enumerate}
				\item If
				\begin{equation}
				\mathcal{K}_\ro\left((A-\sigma E)^{-1}E, (A-\sigma E)^{-1} B\,r\right) \subseteq \Image(V), \label{eq:input_Krylov_higherOrder}
				\end{equation}
				then $G^{(i)}(\sigma) \, r = G_r^{(i)}(\sigma)\, r$ for $i=0,\dots,\ro-1$.
				\item If
				\begin{equation}
				\mathcal{K}_\ro\left(\invt{(A-\sigma E)}\trans{E}, \invt{(A-\sigma E)} \trans{C}\,l\right) \subseteq \Image(W), \label{eq:output_Krylov_higherOrder}
				\end{equation}
				then $\trans{l} G^{(i)}(\sigma) = \trans{l} G_r^{(i)}(\sigma)$ for $i=0,\dots,\ro-1$.
				\item If both \cref{eq:input_Krylov_higherOrder} and \cref{eq:output_Krylov_higherOrder} hold, then, in addition,
				\begin{equation}
				\trans{l}\, G^{(i)}(\sigma) \, r= \trans{l}\, G^{(i)}(\sigma) \, r,
				\qquad i = \ro,\,\dots\,,\,2\ro-1.
				\end{equation}
			\end{enumerate}
		\end{theorem}
		
		
		For bitangential Hermite interpolation as of \cref{eq:tangential_interpolation_conditions}, any bases $V$ and $W$ satisfying \cref{eq:input_Krylov} and \cref{eq:output_Krylov}, respectively, can be selected. 
		For theoretical considerations, \emph{primitive} bases are of particular interest, as defined in the following.
		
		\begin{definition}[Primitive bases]\label{def:primitive basis}
			Consider a full-order model $\Sigma$ as in \cref{eq:FOM}. Let interpolation frequencies $\sigma_i\tin\Complex$ and tangential directions $r_i\tin\Complex^{\is}$ and $l_i\tin\Complex^{\os}$ be given such that $A-\sigma_iE$ is invertible for all $i\teq1,\dots,\ro$.
			Then the \emph{primitive} projection bases $V^\prim$, $W^\prim$ are defined as
					\begin{subequations}
						\begin{align}
							V^\prim &= \left[(A-\sigma_1E)^{-1}Br_1\,,\, \dots\,,\,(A-\sigma_\ro E)^{-1}Br_\ro  \right] \label{eq:Vprim}\\
							W^\prim &= \left[\invt{(A-\sigma_1E)}\trans{C}l_1\,,\, \dots\,,\,\invt{(A-\sigma_\ro E)}\trans{C}l_\ro  \right] \label{eq:Wprim}
						\end{align}
						\label{eq:PrimBases}
					\end{subequations}
		\end{definition}
		On the other hand, from a numerical standpoint, $V$ and $W$ should be preferably orthogonal (or bi-orthogonal), real bases, improving the conditioning and resulting in a $\Sigma_R$ with real matrices. 
		Provided the frequencies $\shiftSetInput$ and respective tangential directions $\tanDirSetInput$, $\tanDirSetOutputI$ are closed under conjugation,
		this is always possible through bases changes $V \teq V^\prim T_V$ and $W \teq W^\prim T_W$, with regular $T_V,T_W\tin\Complex^{\ro\times\ro}$.
		
		Finally, note that the primitive bases \cref{eq:PrimBases} can be defined as solutions of particular \emph{generalized dense-sparse Sylvester equations}.
		\begin{lemma}[\cite{morGalVV04a,morGalVV04}] \label{thm:Sylvester equivalence}
			Consider a full-order model $\Sigma$ as in \cref{eq:FOM}. 
			Let interpolation frequencies $\sigma_i\tin\Complex$ and tangential directions $r_i\tin\Complex^{\is}$ be given such that $A-\sigma_iE$ is invertible for all $i\teq1,\dots,\ro$.
			Define matrices $\Sprim{}\defeq \diag\left(\sigma_1,\,\dots\,,\,\sigma_\ro\right)$, $\Rprim\defeq\left[r_1,\,\dots\,,\,r_\ro\right]$.
			Then the primitve basis $\Vprim$ satisfying \cref{eq:Vprim} solves the generalized sparse-dense Sylvester equation
			\begin{equation}
			A \Vprim - E \Vprim \Sprim{} - B\Rprim = 0 \label{eq:Sylvester:V}.
			\end{equation}
		\end{lemma}
		
		This relationship is particularly useful for theoretical considerations and will be exploited in \cref{sec:aNewFramework} in the proofs. A dual result for $\Wprim$ holds as well.
		\begin{lemma}\label{thm:Sylvester equivalence:W}
			Consider a full-order model $\Sigma$ as in \cref{eq:FOM}. 
			Let interpolation frequencies $\sigma_i\tin\Complex$ and tangential directions $l_i\tin\Complex^{\os}$ be given such that $A-\sigma_iE$ is invertible for all $i=1,\dots,\ro$.
			Define matrices $\Sprim{}\defeq \diag\left(\sigma_1,\,\dots\,,\,\sigma_\ro\right)$, $\Lprim\defeq\left[l_1,\,\dots\,,\,l_\ro\right]$.
			Then the primitve basis $W^\prim$ satisfying \cref{eq:Wprim} solves the generalized sparse-dense Sylvester equation
			\begin{equation}
			\trans{A} \Wprim - \trans{E} \Wprim \Sprim{}- \trans{C}\Lprim = 0 \label{eq:Sylvester:W}.
			\end{equation}
		\end{lemma}
	\subsection{\HtwoText-Optimal Reduction}
		For the design of $V$ and $W$ as of \cref{sec:tangentialInterpolation}, an appropriate choice for the interpolation frequencies $\shiftSetInput$ and tangential directions $\tanDirSetInput$ and $\tanDirSetOutputI$ needs to be made. 
		This is a non-trivial task, as the inspection of the FOM is a computationally challenging task in the large-scale setting. For this reason an automatic selection of reduction parameters minimizing the approximation error $\norm{G-G_r}{}$ for some chosen norm is highly desirable.
		
		In this contribution, we address the problem of finding an optimal ROM of prescribed order $\ro$ that minimizes the error measured in the \HtwoText\ norm, i.e.
		\begin{equation}
			G_r = \arg\min_{\deg\widehat{G}_r=\ro}\norm{G-\widehat{G}_r}{\Htwo},
			\label{eq:H2 optimization problem}
		\end{equation}	
		where the \HtwoText\ norm is defined as \cite{morAnt05}
		\begin{equation}
			\norm{G}{\Htwo} \defeq \left(\frac{1}{2\pi}\int\limits_{-\infty}^{\infty}\trace\left( \htrans{G}(-j\omega)G(j\omega)\right)\mathrm{d}\omega\right)^{1/2}.
		\end{equation}
		Note that there exists a direct relation between the approximation error in the frequency domain in terms of the \HtwoText\ norm and a bound for the $L_\infty$ norm of the output error in the time domain, according to \cite{morAntBG10}
		\begin{equation}
			\norm{y-y_r}{L_\infty} \leq \norm{G-G_r}{\Htwo}\,\norm{u}{L_2}.
		\end{equation}

		The optimization problem \eqref{eq:H2 optimization problem} is non-convex, therefore in general only local optima can be found. 
		Necessary conditions for local \HtwoText-optimality in terms of bitangential Hermite interpolation are available. 
		\begin{theorem}[\cite{Meier_Luenberger_1967,morGugAB08,morVanGA08}]\label{thm:H2optimality}
			Consider a full-order model \cref{eq:FOM} with transfer function $G(s)$. 
			Consider a reduced-order model with transfer function $G_r(s)\teq\sum\limits_{i=1}^{\ro} \frac{\outRes\inRes}{s-\lambda_{r,i}}$ with reduced poles $\lambda_{r,i}\tin\Complex$ and input resp. output residual directions $\trans{\inRes}\tin\Complex^{\is}$, $\outRes\tin\Complex^{\os}$. 
			
			If $G_r(s)$ satisfies \cref{eq:H2 optimization problem} locally, then
			\begin{subequations}
				\begin{align}
					G(-\bar{\lambda}_{r,i})\trans{\inRes} &= G_r(-\bar{\lambda}_{r,i})\trans{\inRes} \label{eq:OC1}\\
					\trans{\outRes}G(-\bar{\lambda}_{r,i}) &= \trans{\outRes}G_r(-\bar{\lambda}_{r,i}) \label{eq:OC2}\\
					\trans{\outRes}G'(-\bar{\lambda}_{r,i})\trans{\inRes} &= \trans{\outRes}G'_r(-\bar{\lambda}_{r,i})\trans{\inRes} \label{eq:OC3}
				\end{align}
				\label{eq:OC}
			\end{subequations}
			for $i=1,\,\dots\,,\ro$.
		\end{theorem}
		The extension to the case of poles with higher multiplicities is omitted here for brevity and can be found in \cite{morVanGA10}.
	
		\cref{thm:tangential_interpolation} indicates how to construct bitangential Hermite interpolants for given interpolation data. However, it is not possible to know a-priori the eigenvalues and residual directions of the reduced order model. 
		For this reason, an iterative scheme known as Iterative Rational Krylov Algorithm (IRKA) has been developed \cite{morGug08,morVanGA08,morBeaG14} to iteratively adapt the interpolation data until the conditions \cref{eq:OC} are satisfied.
		A sketch is given in \cref{algo:IRKA}.
		
		 \begin{algorithm*}[!ht]\caption{MIMO \HtwoText-Optimal Tangential Interpolation (IRKA)} \label{algo:IRKA}
		 	\begin{algorithmic}[1]
		 		\Require FOM $\Sigma$; Initial interpolation data $S^\prim$, $R^\prim$, $L^\prim$ 
		 		\Ensure locally $\mathcal{H}_2$-optimal reduced model $\Sigma_r$
		 		\While{not converged} 
		 		\State{$\Vprim \gets A\Vprim-E\Vprim\Sprim{} - B\Rprim = 0$ \hfill{// compute projection matrix \cref{eq:Vprim}}} \label{algo:irka:V}
		 		\State{$\Wprim \gets \trans{A}\Wprim-\trans{E}\Wprim\Sprim{} - \trans{C}\Lprim = 0$ \hfill{// compute projection matrix \cref{eq:Wprim}}}\label{algo:irka:W}
		 		\State{$V \gets \text{qr}(\Vprim)$; $W \gets \text{qr}(\Wprim)$ \hfill{// compute orthonormal bases}}\label{algo:irka:qr}
		 		\State{$\Sigma_r \gets \trans{W}\Sigma V$ \hfill{// compute reduced model by projection}}
		 		\State{[$X,D,Y$] = eig($\Sigma_r$) \hfill{// eigendecomposition}}
		 		\State{$S^\prim \gets - \overline{D}; R^\prim \gets \trans{B_r}Y;  L^\prim \gets C_rX$ \hfill{// update interpolation data}}
		 		\EndWhile
		 	\end{algorithmic}
		 \end{algorithm*}
		 
		 Note that the primitive bases in lines \ref{algo:irka:V} and \ref{algo:irka:W} are not computed solving Sylvester equations (as in \cite{morXuZ11}) but rather the sparse linear systems in \eqref{eq:PrimBases}. 
		 The equivalent representation using Sylvester equations (cp. \cref{thm:Sylvester equivalence}) is used for brevity.
		 
		 Finally, note that \cref{algo:IRKA} is not the only \HtwoText-optimal method for linear systems present in literature, but is certainly best known due to its simplicity and effectiveness. 
		 Other approaches worth mentioning include the trust-region algorithms in \cite{morBeaG09b} and \cite{morPanJWetal13}.
		 In addition, \cite{morBeaG12} derives a \emph{residue correction} algorithm to optimize the tangential directions for fixed poles, speeding up convergence for models with many inputs and outputs.
		 Even though for brevity we will not treat all these algorithms individually, the proofs of \cref{sec:aNewFramework} will make evident that the new framework presented in this paper applies to these algorithms as well, as all methods are targeted at satisfying the optimality conditions \eqref{eq:OC}.


\section{The Cost of \HtwoText-Optimal Reduction}\label{sec:costOfH2}
	The computational cost of model reduction by IRKA (cp. \cref{algo:IRKA}) is dominated by the large-scale linear systems of equations (LSE) involved in computing $\Vprim$, $\Wprim$ according to
	\begin{subequations}%
		\begin{align}%
			\overbrace{\left(A -\sigma_i E\right)}^{=:\Asig} V_i^\prim &= B \,r_i, \quad i=1,\,\dots\,,\ro, \label{eq:LSE:V} \\%
			\trans{\Asig} W_i^\prim &= \trans{C} \,l_i \quad i=1,\,\dots\,,\ro. \label{eq:LSE:W}%
		\end{align}%
		\label{eq:LSE}%
	\end{subequations}%
	In fact, the orthogonalization process involved in transforming $V\teq \Vprim T_V$ and $W\teq\Wprim T_W$, as well as the matrix-matrix multiplications involved in the projection $\Sigma_r\teq\trans{W}\Sigma V$ and the low-dimensional eigenvalue decomposition are in general of subordinated importance\footnote{For dense matrices, this can be motivated by simple asymptotic operation counts. The QR decomposition of a $V\ts\in\ts\Reals^{\fo\times\ro}$ matrix via Householder requires $2\ro^2\left(\fo-\frac{\ro}{3}\right)$ flops and is hence linear in $\fo$. The flops involved in the product $\trans{W}EV$ are $\ro\fo(2\fo-1) + \ro^2(2\fo-1)$ for a dense $E$ and hence quadratic in $\fo$. Note however that for a diagonal $E$ matrix---an ideally sparse invertible matrix---the flops become at most $\ro\fo + \ro^2(2\fo-1)$, hence being linear in $\fo$ \cite{golub1996matrix}.}. 
	
	As the matrix $\Asig$ of large-scale systems is in general \emph{sparse} \cite{saad2003iterative,davis2006direct}, the actual cost involved in solving one LSE depends on a series of factors (including sparsity pattern, number of nonzero elements, conditioning, \dots) as well as the effective exploitation of available hardware resources. 
	It is therefore not possible (or even meaningful) to perform asymptotic operation counts as in the dense case.
	Nonetheless, to demonstrate that the reduction cost is indeed dominated by the solution of the sparse LSE in \cref{eq:LSE}, \cref{fig:algo complexity} compares the average execution times for the different computation steps of \cref{algo:IRKA} using MATLAB\textsuperscript{\textregistered} R2016b on an Intel\textsuperscript{\textregistered} Core\textsuperscript{\texttrademark} i7-2640 CPU @ \unit[2.80]{GHz} computer with \unit[8]{GB} RAM\footnote{Unless otherwise stated, this setup will be used for all numerical results.}. 
	
	The comparison includes the sparse \mcode{lu} decomposition of the matrix $A_{\sigma=1}$,  the economy-sized \mcode{qr} decomposition of the projection matrix $\Vprim$, the matrix products involved in $\trans{W} \Sigma V$ as well as the small dimensional generalized \mcode{eig} decompositions. 
	The reduced order is set to $\ro\teq10$ for all cases, while the original model order is given on the x-axis. 
	The times given are averaged amongst several executions and cumulated for each IRKA step: At each iteration of IRKA, \ro\ \mcode{lu} decompositions, two \mcode{qr} decompositions, one projection $\trans{W}\Sigma V$ and one \mcode{eig} decomposition are performed.
	The models used are taken from the benchmark collections \cite{morChaV02,morKorR05}.
	\setlength{\mywidth}{12cm}
	\setlength{\myheight}{3cm}
	\begin{figure}[h]
		\centering
%
\definecolor{mycolor1}{rgb}{0.76863,0.02745,0.10588}%
\definecolor{mycolor2}{rgb}{0.63529,0.67843,0.00000}%
\definecolor{mycolor3}{rgb}{0.00000,0.39608,0.74118}%
\definecolor{mycolor4}{rgb}{0.89020,0.44706,0.13333}%
\begin{tikzpicture}

\begin{axis}[%
width=0.816\mywidth,
height=\myheight,
at={(0\mywidth,0\myheight)},
scale only axis,
xmode=log,
xmin=30,
xmax=200000,
xminorticks=true,
xlabel style={font=\color{white!15!black}},
xlabel={$\fo$},
ymode=log,
ymin=1e-05,
ymax=1000,
yminorticks=true,
ylabel style={font=\color{white!15!black}},
ylabel={execution time /$s$},
axis background/.style={fill=white},
xmajorgrids,
xminorgrids,
ymajorgrids,
yminorgrids,
legend style={at={(1.03,0.5)}, anchor=west, legend cell align=left, align=left, draw=white!15!black}
]
\addplot [color=mycolor1, mark=o, mark options={solid, mycolor1}]
  table[row sep=crcr]{%
48	0.0059973345045816\\
120	0.00363129225210213\\
270	0.00962660424034135\\
348	0.0796647150893303\\
1006	0.00504406519747457\\
11730	1.23675594416044\\
20082	9.97163662324294\\
34722	8.35126888751036\\
66917	139.442065049559\\
79841	3.97809904594584\\
106437	45.7537055970876\\
};
\addlegendentry{\mcode{lu}}

\addplot [color=mycolor2, dashed, mark=triangle, mark options={solid, mycolor2}]
  table[row sep=crcr]{%
48	0.000160243667939041\\
120	0.000172113122055843\\
270	0.000599510068056166\\
348	0.000528848780679233\\
1006	0.00071202838767714\\
11730	0.00301606693019753\\
20082	0.00591295633024975\\
34722	0.00973269276920221\\
66917	0.0200722732630886\\
79841	0.0243353815084651\\
106437	0.0343398401004617\\
};
\addlegendentry{\mcode{qr}}

\addplot [color=mycolor3, dashed, mark=x, mark options={solid, mycolor3}]
  table[row sep=crcr]{%
48	0.000114797423857542\\
120	8.93831476089782e-05\\
270	0.000277142093847777\\
348	0.00070424622778062\\
1006	0.000402689041131039\\
11730	0.0035506391302576\\
20082	0.00445343910791936\\
34722	0.0192197373445568\\
66917	0.0208684081437979\\
79841	0.0215857795970363\\
106437	0.0562527541691759\\
};
\addlegendentry{$\trans{W}\Sigma V$}

\addplot [color=mycolor4, dashed, mark=triangle, mark options={solid, rotate=180, mycolor4}]
  table[row sep=crcr]{%
48	0.000188676625295642\\
120	4.74415922939109e-05\\
270	0.000748427644477851\\
348	0.000212859279652127\\
1006	0.000199866876163136\\
11730	0.000249666058402423\\
20082	0.00023686081314213\\
34722	0.000301373047101991\\
66917	0.000613691831901724\\
79841	0.000212367234631517\\
106437	0.000173827732925208\\
};
\addlegendentry{\mcode{eig}}

\end{axis}
\end{tikzpicture}%
		\caption{Cumulated execution times involved in each step of \cref{algo:IRKA} for different benchmark models ($\ro\teq10$).}
		\label{fig:algo complexity}
	\end{figure}
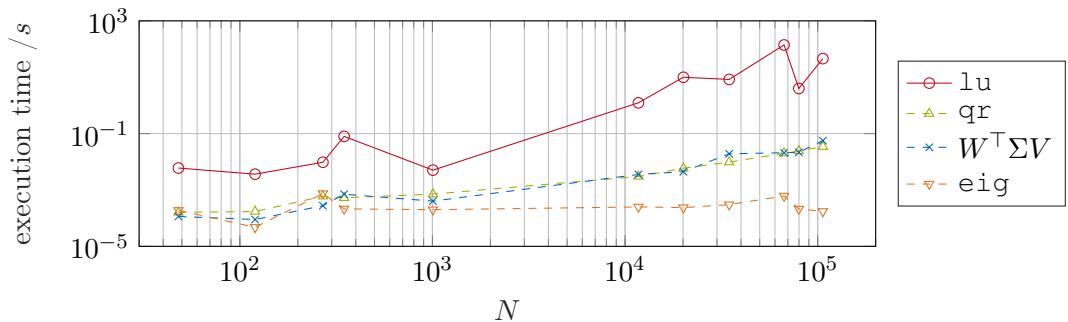 
	
	As it can be seen, the execution time for the sparse \mcode{lu} decompositions grows more and more dominant as the problem size increases. 
	This becomes even more evident in \cref{fig:pie}, where the execution times are given as percentage of the total time for one IRKA iteration. The two models shown represent the extreme cases of \cref{fig:algo complexity}, i.e. where the execution time for the \mcode{lu} decompositions has the smallest and largest share. 
	\begin{figure}[h]
		\begin{subfigure}[b]{.5\linewidth}
			\centering
			\begin{tikzpicture}
			\pie[radius =2, text=legend,explode ={0.2 , 0, 0, 0}]{79.3/\mcode{lu},11.2/\mcode{qr}, 6.3/$\trans{W}\Sigma V$,3.1/\mcode{eig}}
			\end{tikzpicture}
			\caption{FOM model ($\fo=1006$)}
			\label{fig:pie:1}
		\end{subfigure}%
		\begin{subfigure}[b]{.5\linewidth}
			\centering
			\begin{tikzpicture}
			\pie[radius =2, text=legend,explode ={0.2 , 0, 0, 0}]{100.0/\mcode{lu},0.0/\mcode{qr}, 0.0/$\trans{W}\Sigma V$,0.0/\mcode{eig}}
			\end{tikzpicture}
			\caption{Gas Sensor model ($\fo=66917$)}
			\label{fig:pie:2}
		\end{subfigure}		
		\caption{Cumulated execution times as total share of each step of \cref{algo:IRKA}.}
		\label{fig:pie}
	\end{figure}
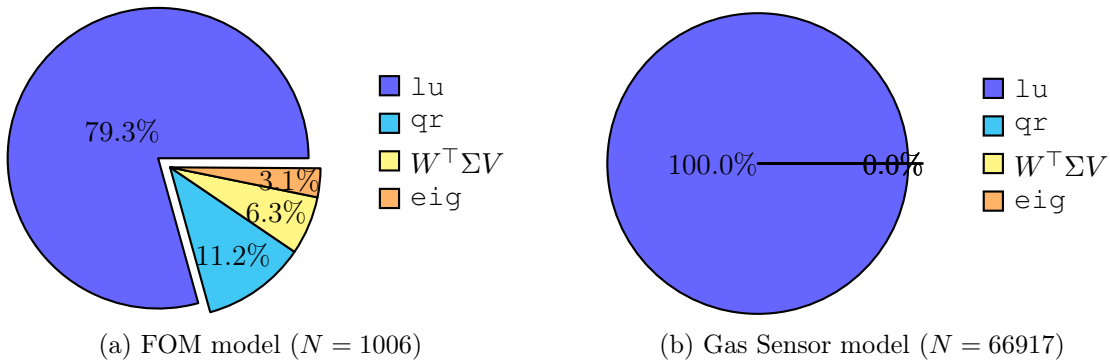

	IRKA (and alternative \HtwoText\ reduction methods) require the repeated reduction of an $\fo$\textsuperscript{th}-order model until, after $\kIrka$ steps, a set of \HtwoText-optimal parameters $\shiftSetInputOpt$, $\tanDirSetInputOpt$, and $\tanDirSetOutputIOpt$ is found.
	Following the results of \cref{fig:algo complexity} and \cref{fig:pie}, its cost can be approximated by 
	\begin{equation}
		\Cost{\text{IRKA}}{\fo} \approx \underbrace{\kIrka}_{optimization} \cdot \,\underbrace{2\ro\, \Cost{\text{LSE}}{\fo}}_{reduction},
		\label{eq:Cost of IRKA}
	\end{equation}	
	where the cost of a $\fo$-dimensional LSE $\Cost{\text{LSE}}{\fo}$ depends on the model at hand as well as the chosen LSE solver\footnote{Note that complex conjugated pairs of shifts $\sigma_i\teq\overline{\sigma_j}$ yield complex conjugated directions $\Vprim_i\teq\overline{\Vprim}_j$. Therefore, if direct solvers are used, then the factor $2\ro$ reduces to $\ro$ if the LU factorization of \cref{eq:LSE:V} is recycled in \cref{eq:LSE:W}. The factor $2\ro$ can thus be considered as a worst-case scenario.}.
	While the second factor in \cref{eq:Cost of IRKA} represents the \emph{cost of a single reduction}, the factor $\kIrka$ represents the cost introduced by the \emph{optimization}.
	From this representation, it becomes evident that the cost of optimization is tied to---in fact weighted with---the cost of a full reduction. 
	The efficiency of \HtwoText-optimal reduction methods can hence be significantly deteriorated by bad convergence, which can be a result of a bad initialization or the selection of an unsuitable reduced order.
	A very similar discussion applies to trust-region-based \HtwoText-optimal reduction methods \cite{morBeaG09b,morPanJWetal13}, where the evaluation of gradient and Hessian in each step also involves a full reduction.
	
	Clearly, a more desirable setting would be to obtain \HtwoText-optimal reduction parameters at a far lower cost than the cost of reduction, having to reduced the full-order model only once.
	In the following, we introduce a new framework that effectively \emph{decouples} the cost of reduction from the cost of finding \HtwoText-optimal reduction parameters without compromising optimality.

\section{A New Framework for \HtwoText-Optimal Reduction}\label{sec:aNewFramework}
	The framework discussed in this section was first introduced in \cite[p.83]{morPan14} for SISO models, under the name of \emph{model function}, a heuristic to reduce the cost of \HtwoText-optimal reduction within the SPARK algorithm. 
	This heuristic was later applied to IRKA and proven in \cite{morCasPL17} to yield \HtwoText-optimal reduced order models under certain \emph{update conditions}. 
	In this contribution, we give a more extensive discussion of the framework, extending its validity to MIMO models. In \cref{sec:furtherApplications} we will indicate how to apply this framework also to further system classes for which \HtwoText\ approximation algorithms are available.
	
	The main motivation for the new framework arises from simple considerations on the \emph{locality} of model reduction by tangential interpolation. In fact,
	\begin{enumerate}
		\item tangential interpolation only guarantees to yield a good approximation \emph{locally} around the frequencies $\shiftSetInput$, tangentially along directions $\tanDirSetInput$ and $\tanDirSetOutputI$,
		\item as $\normHtwo{G-G_r}$ is a non-convex function, in general only \emph{local} optima can be achieved.
	\end{enumerate}
	
	This can be exploited during reduction as follows:
	Suppose a full-order model $\Sigma$ and initial tangential interpolation data $\shiftSetInputK{0}$,$\tanDirSetInputK{0}$,  and $\tanDirSetOutputIK{0}$ are given. 
	Conventional \HtwoText\ reduction approaches would initialize e.g. IRKA and run for $\kIrka$ steps until convergence.
	In contrast, as our goal is to find a \emph{local} optimum close to the initialization, then optimization with respect to a surrogate model---a good \emph{local} approximation with respect to the initial interpolation data---may suffice.
	
	For this reason, the new framework starts by building an intermediate model $\SigmaM{}$---in the following denoted as \emph{Model Function}, in accordance to its first introduction in \cite{morPan14}---of order $\nM{}$, with $\fo\gg\nM{}>\ro$. 
	\begin{definition}\label{def:model function}
		Consider a full-order model $\Sigma$ as in \cref{eq:FOM}. 
		Let interpolation data $\shiftMi\tin\Complex$, $\tanDirInputMi\tin\Complex^{\is}$, $\tanDirOutputMi\tin\Complex^{\os}$, $i\teq1,\dots,\nM{}$, be given and define the matrices 
		$\SModel{}\teq\diag \left(\sigma_{\model,1}, \dots, \sigma_{\model,\nM{}}\right)$,
		$\RModel{}\teq\begin{bmatrix} r_{\model,1}&\dots&r_{\model,\nM{}}	\end{bmatrix}$, 
		$\LModel{}\teq\begin{bmatrix} l_{\model,1}&\dots&l_{\model,\nM{}}	\end{bmatrix}$. 
		Then the Model Function $\SigmaM{}$ is defined as
		\begin{equation}
			\SigmaM{} \defeq \WM{}\Sigma \VM{},
			\label{eq:model function}
		\end{equation}
		where $\VM{}, \WM{}$ solve the Sylvester equations
		\begin{subequations}
			\begin{align}
				A\VM{} - E\VM{} \SModel{} - B\RModel{} &= 0, \\
				\trans{A}\WM{}-\trans{E}\WM{}\SModel{} - \trans{C}\LModel{} &= 0.
			\end{align}
		\end{subequations}
	\end{definition}
	From \cref{def:model function} and the results of \cref{thm:Sylvester equivalence}, \cref{thm:Sylvester equivalence:W}, and \cref{thm:tangential_interpolation} follows that $\SigmaM{}$ is a bitangential Hermite interpolant of $\Sigma$ with interpolation frequencies and tangential directions specified in $\SModel{}$, $\RModel{}$ and $\LModel{}$. 
	
	\subsection{The Model Function Framework}
	
		\paragraph{Step 1: Initialization of $\SigmaM{}$}
		The goal of the Model Function $\SigmaM{}$ is to be a good approximation of $\Sigma$ locally with respect to the initial interpolation data $\shiftSetInputK{0}$,$\tanDirSetInputK{0}$,  and $\tanDirSetOutputIK{0}$ for it to be used as a \emph{surrogate} of $\Sigma$ during the \HtwoText\ optimization. 
		For obvious reasons it must hold $\nM{}\ts>\ts\ro$, making it not possible to choose $\SModel{}\teq S^0 := \diag\left(\sigma^0_1,\dots,\sigma^0_\ro\right)$, 
		$\RModel{}=R^0:=\begin{bmatrix} r^0_{1}&\dots&r_{\ro}^0	\end{bmatrix}$, and
		$\LModel{}=L^0:=\begin{bmatrix} l^0_{,1}&\dots&l^0_{\ro}\end{bmatrix}$. 
		As an appropriate choice for $\SModel{}$ is not unique, we provide two possible selections that appear to be meaningful to us:
		\begin{enumerate}
			\item [\bf I.1] \label{init:1}Include the initial interpolation data $S^0,R^0,L^0$ with additional interpolation of $G(s)$ and its first $\nM{}\ts-\ts\ro\ts-\ts1$ derivatives at the frequency $\sigma\teq0$ for a sum of all input and output channels. This can be achieved through
			\begin{equation}
				\SModel{} = \left[\begin{array}{c|c c c c}
					S^0 & 0 \\
					\hline
					0   & J_0
				\end{array}\right], \quad 
				\RModel{} = \left[\begin{matrix}
				R^0 & e_1 & 0 &\dots & 0
				\end{matrix}\right], \quad
				\LModel{} = \left[\begin{matrix}
				L^0 & e_1 & 0& \dots & 0
				\end{matrix}\right],
			\end{equation}
			where $J_0$ is a Jordan block of size $\nM{}\ts-\ts\ro$ with eigenvalue $0$ and $e_1 := \trans{\left[\begin{matrix} 1 & 1 & \dots & 1
			\end{matrix}\right]}$ of appropriate dimensions.
			In this case, the choice of $\nM{}$ is free. Due to \cref{thm:moment_matching}, we expect the approximation quality of $\SigmaM{}$ to improve locally around $\sigma\teq0$ as $\nM{}$ grows.
			\item [\bf I.2] \label{init:2} Tangentially interpolate higher-order derivatives with respect to the data $S^0,R^0,L^0$. This can be achieved through
			\begin{equation}
			\SModel{} = \left[\begin{array}{c c c c | c c c c}
			S^0 & J_0 \\
			0 			   & S^0
			\end{array}\right], \quad 
			\RModel{} = \left[\begin{matrix}
			R^0 & 0
			\end{matrix}\right], \quad
			\LModel{} = \left[\begin{matrix}
			L^0 & 0
			\end{matrix}\right],
			\end{equation}
			(cp. \cref{thm:moment_matching}). In this case $\nM{}\teq2\,\ro$ and we expect the approximation quality of $\SigmaM{}$ to increase locally around the frequencies $\shiftSetInputK{0}$ with respect to the tangential directions $\tanDirSetInputK{0}$ and $\tanDirSetOutputIK{0}$.
		\end{enumerate}
		Both approaches bear the advantage of minimizing the additional cost tied to the initialization of $\SigmaM{}$. In fact, note that for every value $\sigma_{\model,i}\teq\sigma_{\model,j}$, the LSEs in \cref{eq:LSE} share the same left hand side, allowing e.g. the recycling of LU decompositions or preconditioners.
		
		\paragraph{Step 2: \HtwoText\ optimization with respect to $\SigmaM{}$}
		As the Model Function $\SigmaM{}$ is a good approximation of $\Sigma$ locally with respect to the initial interpolation data, we can run the \HtwoText\ optimization with respect to $\SigmaM{}$ with initialization $S^0$, $R^0$, $L^0$. 
		The optimal interpolation data found at convergence is denoted by 
		$\SOpt{}\teq\diag\left(\sigma_{\opt,1},\dots,\sigma_{\opt,\ro}\right)$,
		$\ROpt{}\teq\begin{bmatrix} r_{\opt,1}&\dots&r_{\opt,\ro}	\end{bmatrix}$, and
		$\LOpt{}\teq\begin{bmatrix} l_{\opt_,1}&\dots&l_{\opt,\ro}\end{bmatrix}$. 
		As typically $\nM{}\ll\fo$, we are expecting the cost of this optimization to be significantly lower than a full optimization over $\Sigma$.
		As we assume the validity of the resulting reduced order model $\SigmaMr{}$ to be \emph{confined} locally to the frequency region defined by $\SModel{}$, we call the combination of IRKA with this Model Function framework \emph{Confined IRKA} (CIRKA). 
		
		In fact, by approximating $\Sigma$ through $\SigmaM{}$, we have lost any optimality condition between $\Sigma$ and $\SigmaMr{}$ and can only claim optimality conditions between $\SigmaM{}$ and $\SigmaMr{}$.
		Whether $\SigmaMr{}$ is an acceptable approximation of $\Sigma$ highly depends on the approximation quality of $\SigmaM{}$ itself. 
		This loss of connection between $\Sigma$ and $\SigmaMr{}$ is a typical drawback of so-called \emph{two-step} model reduction approaches (cp. e.g. \cite{morLehE07}).
		
		Nonetheless, by means of surrogate optimization, a new set of interpolation data $\SOpt{}$, $\ROpt{}$, $\LOpt{}$ has been found in a cost-effective way. 
		At this point, an update of $\SigmaM{}$ with local information about these frequencies is required.
		
		\paragraph{Step 3: Update of $\SigmaM{}$ and fix-point iteration}
			To restore the relationship between $\Sigma$ and $\SigmaMr{}$, the Model Function $\SigmaM{k}$ from the current step $k$ can be updated to a new $\SigmaM{k+1}$ by enforcing tangential interpolation with respect to the optimal data $\SOpt{k}$, $\ROpt{k}$, $\LOpt{k}$. 
			This can be achieved by updating the projection matrices 
			\begin{subequations}
				\begin{align}
					\VM{k+1} &= \begin{bmatrix} \VM{k} & \VM{k,+} \end{bmatrix}, \label{eq:Update:V}\\
					\WM{k+1} &= \begin{bmatrix} \WM{k} & \WM{k,+} \end{bmatrix}, \label{eq:Update:W}
				\end{align}
				\label{eq:Update}
			\end{subequations}
			with new directions $\VM{k,+}$ and $\WM{k,+}$.
			
			As for the initialization of $\SigmaM{0}$, several strategies are possible to update $\SigmaM{k}$. In the following, we present only a few most relevant ones:
			\begin{enumerate}
				\item [\bf U.1] \label{update:1} Update $\SigmaM{k}$ with tangential interpolation with respect to \emph{all} optimal frequencies and tangential directions $\SOpt{k}$, $\ROpt{k}$, $\LOpt{k}$. 
				This implies that $\nM{k}$ increases by $\ro$ in every step and that higher-order derivatives are tangentially interpolated in case specific combinations of optimal frequencies and tangential directions are repeated.
				In this approach, we expect the approximation quality of $\SigmaM{k+1}$ to increase around all optimal frequencies $\shiftSetInputOpt$, with respect to the tangential directions $\tanDirSetInputOpt$ and $\tanDirSetOutputIOpt$.
				\item [\bf U.2] \label{update:2} Update $\SigmaM{k}$ only with tangential interpolation with respect to \emph{new} interpolation data. This implies that the order of the Model Function increases by at most $\ro$ in every step.
				In this approach, we expect the approximation quality of $\SigmaM{k+1}$ to increase only around new optimal frequencies $\shiftOpti$ with respect to the respective tangential directions $\tanDirInputOpti$ and $\tanDirOutputOpti$, for $i\teq 1,\dots,\nM{k,+}$, where $\nM{k,+}$ denotes the number of new frequencies.
				\item [\bf U.3] \label{update:3} \emph{Reinitialize} $\SigmaM{k}$ by including ``only'' tangential interpolation about the new optimal data.
				In this approach, we expect the approximation quality of $\SigmaM{k+1}$ to increase in the regions around the new optimal frequencies $\shiftSetInputOpt$, with respect to the tangential directions $\tanDirSetInputOpt$ and $\tanDirSetOutputIOpt$, but potentially decrease around previous interpolation data.
				However, the oder $\nM{k}$ can be kept constant through iterations using this method, reducing the cost of \HtwoText\ optimization.
				As $\nM{k}>\ro$, similar considerations apply as for the Initialization in Step 1.
			\end{enumerate}
			
			The updated Model Function $\SigmaM{k+1}$ can be used again to perform a low-dimensional \HtwoText\ optimization and potentially improve the optimal frequencies and directions $\SOpt{k+1}$,$\ROpt{k+1}$ and $\LOpt{k+1}$, which in general may differ from those of the previous step.
			This leads to a fix-point iteration, until, after $\kCirka$ steps, an update of $\SigmaM{\kCirka-1}$ does not result in new optimal interpolation data.
			
			The overall procedure is summarized in \cref{algo:CIRKA} for the case where \HtwoText\ optimization is performed with IRKA. Note however that \cref{algo:CIRKA:IRKA} can be replaced by any \HtwoText-optimal reduction method, leading to the more general \emph{Model Function framework}.
			
			\begin{algorithm*}[!ht]\caption{Confined IRKA (CIRKA)} \label{algo:CIRKA}
				\begin{algorithmic}[1]
					\Require FOM $\Sigma$; Initial interpolation data $S^0$, $R^0$, $L^0$ 
					\Ensure  reduced model $\SigmaMr{}$, Model Function $\SigmaM{\kCirka}$, error estimation $\errorEst$
					\State{$k\gets0$; $\left[\SigmaM{k},\SOpt{\text{tot}},\ROpt{\text{tot}},\LOpt{\text{tot}}\right]\gets$ empty; \hfill{// Initialization}}
					\State{$\SOpt{k}\gets S^0$, $\ROpt{k}\gets R^0$, $\LOpt{k}\gets L^0$; }
					\While{not converged} 
						\State{$k\gets k+1$}
						\State{$\SigmaM{k}\gets$updateModelFunction$\left(\Sigma,\SigmaM{k-1},\SOpt{\text{tot}},\ROpt{\text{tot}},\LOpt{\text{tot}},\SOpt{k-1},\ROpt{k-1},\LOpt{k-1}\right)$}
						\State{$\left[\SigmaMr{},\SOpt{k},\ROpt{k},\LOpt{k} \right]\gets$ IRKA $(\SigmaM{k},\SOpt{k-1},\ROpt{k-1},\LOpt{k-1})$\hfill{// \HtwoText\ optimization}} \label{algo:CIRKA:IRKA}
						\State{$\SOpt{\text{tot}}\gets$blkdiag$\left(\SOpt{\text{tot}},\SOpt{k}\right)$; 
							   $\ROpt{\text{tot}}\gets \left[\ROpt{\text{tot}},\ROpt{k}\right]$;
							   $\LOpt{\text{tot}}\gets \left[\LOpt{\text{tot}},\LOpt{k}\right]$ }
					\EndWhile
					\State{$\kCirka\gets k$}
					\State{$\errorEst \gets$norm$\left(G_{\model}^{\kCirka}-G_{\model,r}\right)$}
				\end{algorithmic}
			\end{algorithm*}
			
			For brevity, details on the implementation are omitted at this point. MATLAB\textsuperscript{\textregistered} code for the proposed algorithms is provided with this paper as functions of the \mcode{sssMOR} toolbox \cite{morCasCJetal17}. 

	\subsection{An illustrative example}
	To make the explanation of the framework clearer, we include a simple numerical example to accompany the discussion. 
	As a test model we consider the benchmark model \mcode{beam} \cite{morChaV02} of order $\fo\ts=\ts348$ with $\is\teq\os\teq1$.
	Our aim is to find a model of reduced order $\ro\ts=\ts 4$ satisfying \HtwoText-optimality conditions. 
	We assume that no previous information about the relevant frequency region is given, hence the initialization is chosen with interpolation frequencies $\sigma_1^0\teq\dots\teq\sigma_\ro^0\teq0$, corresponding to $S^0\ts=\ts J_0$. As the model is SISO, there is no need to specify tangential directions. 
	\paragraph{Step 1: Initialization of $\SigmaM{}$}	
	As initial order for the Model Function we select $\nM{}\ts=\ts 2\ro\ts=\ts8$. In this particular case, both initializations I.1 and I.2 of $\SigmaM{}$ coincide. \cref{fig:cirka demo 1} shows the Bode plot of $\Sigma$ and the resulting $\SigmaM{}$\footnote{All numerical examples presented in this contribution were generated using the \mcode{sss} and \mcode{sssMOR} toolboxes in MATLAB\textsuperscript{\textregistered} \cite{morCasCJetal17}.}. 
	The orange circles indicate the imaginary part of the frequencies used to initialize the Model Function.
	\setlength{\mywidth}{8cm}
	\setlength{\myheight}{4cm}
	\begin{figure}[h]
		\centering
		\input{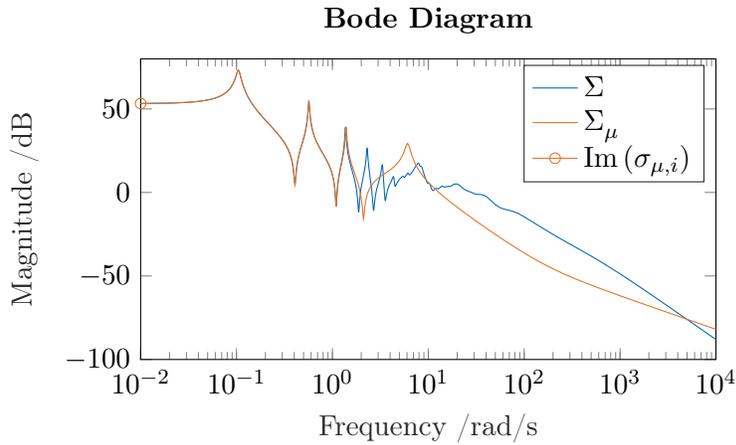}
		\caption{Illustrative Bode plot of $\Sigma$ and its Model Function $\SigmaM{}$.}
		\label{fig:cirka demo 1}
		\end{figure}
	
	\paragraph{Step 2: \HtwoText\ optimization with respect to $\SigmaM{}$}
	\cref{fig:cirka demo 2} shows the reduced order model $\SigmaMr{}$ resulting from the optimization with respect to $\SigmaM{}$.
	\begin{figure}[h]
		\centering
		\input{_pics/cirka_demo_2.tikz}
		\caption{Comparison of $\Sigma$, $\SigmaM{}$ and $\SigmaMr{}$ after \HtwoText\ optimization.}
		\label{fig:cirka demo 2}
	\end{figure}
	In this example, IRKA converged in $\kIrka\teq 4$ steps at optimal frequencies $\SOpt{}\teq\diag\left(\begin{matrix} 0.005 \pm \iu0.104 & 0.006 \pm \iu0.569\ \end{matrix}\right)$, whose imaginary parts are depicted in \cref{fig:cirka demo 2} by green diamonds. 
	In this case, by inspection of the Bode plot, we see that $\SigmaMr{}$ appears to be an acceptable approximation also of the original model $\Sigma$ for the given order $\ro$. This is a result of $\SigmaM{}$ being a valid approximation of $\Sigma$ also locally around the new frequency region represented by $\SOpt{}$.
	However, this need not hold true in general.  For this reason, an update of the Model Function is performed.
	
	\paragraph{Step 3: Update of $\SigmaM{}$ and fix-point iteration}
	Using $\SOpt{1}$, we update $\SigmaM{1}$ to $\SigmaM{2}$ by including interpolation of the full-order model $\Sigma$ with respect to the frequencies in $\SOpt{1}$. As all optimal frequencies differ from those used for initialization of $\SigmaM{}$, update strategies U.1 and U.2 coincide in this case.
	\cref{fig:cirka demo 3} shows the updated Model Function $\SigmaM{2}$ of order $\nM{2}\teq12$ for the \mcode{beam} example, as well as the reduced order model $\SigmaMr{2}$ and new optimal frequencies $\SOpt{2}$ resulting from IRKA on $\SigmaM{2}$. 
	\begin{figure}[h]
		\centering
		\input{_pics/cirka_demo_4.tikz}
		\caption{Comparison of $\Sigma$, $\SigmaM{2}$ and $\SigmaMr{2}$ after repeated \HtwoText\ optimization}
		\label{fig:cirka demo 3}
	\end{figure}
	In this example, IRKA converges (within some tolerance) to the same optimal frequencies $\SOpt{2}\ts\approx\ts\SOpt{1}$ and hence the whole framework required $\kCirka\teq2$ iterations until convergence. The fact that an update of the Model Function does not yield a new optimum indicates that the Model Function $\SigmaM{1}$ was already accurate enough in the region around $\SOpt{1}$.
	
	Finally, note that while this framework has required 3 full-sized LU decompositions to generate and update the Model Function, a direct application of IRKA to $\Sigma$ requires $1\ts+\ts(\kIrka\ts-\ts1)\cdot\ro/2\ts=\ts1\ts+\ts3\cdot 2\ts=\ts7$ full-sized LU decompositions, i.e. more than double the amount. A deeper discussion and comparison of the complexities will be given in \cref{sec:costOfNewFramework} and \cref{sec:numericalresults}.
	It is also worth noting that both CIRKA and IRKA converge to the same local optimum.
	
	\subsection{Optimality of the New Framework}
		So far, the new framework has been presented as a heuristic to perform \HtwoText\ surrogate optimization and hopefully reduce the overall reduction cost.
		The update of $\SigmaM{k}$ has been introduced to increase the accuracy of the surrogate model in the frequency regions the optimizer seemed to deem as important.
		However our original intent was to obtain a reduced model $\Sigma_r$ satisfying the \HtwoText-optimality conditions \cref{eq:OC}. The question still remains as to whether this goal has been achieved.
		
		In this section, we prove that the update of $\SigmaM{k}$ is sufficient to satisfy the optimality conditions \cref{eq:OC} at convergence. 
	
		\begin{theorem}\label{thm:SigmaMr_optimality}
			Consider a full-order model $\Sigma$ as in \eqref{eq:FOM} and let $G(s)$ denote its transfer function.
			Let $\VM{},\WM{}\tin\Complex^{\fo\times\nM{}}$ be projection matrices satisfying 
			Sylvester equations of the form 
			\begin{subequations}
			\begin{align}
				A \VM{} - E \VM{} \SModel{} - B\RModel{} &= 0, \\
				\trans{A} \WM{} - \trans{E} \WM{} \SModel{}- \trans{C}\LModel{} &= 0,
			\end{align}\label{eq:Proof:Sylv1}
			\end{subequations}
			with matrices $\SModel{}\teq\diag\left(\sigma_{\model,1},\dots,\sigma_{\model,\nM{}}\right)$, 
			$\RModel{}\teq\begin{bmatrix}r_{\model,1},\dots,r_{\model,\nM{}}\end{bmatrix}$,
			$\LModel{}\teq\begin{bmatrix}l_{\model,1},\dots,l_{\model,\nM{}}\end{bmatrix}$ and
			${\shiftMj\tin\Complex}$, $\tanDirInputMj\tin\Complex^{\is}$, $\tanDirOutputMj\tin\Complex^{\os}$.
			Consider the Model Function $\SigmaM{}$ resulting from the projection $\SigmaM{}\teq\WM{\top}\Sigma\VM{}$ and let $\GM(s)$ denote its transfer function.
			
			Let $\VMr{}, \WMr{}\tin\Complex^{\nM{}\times\ro}$ be projection matrices satisfying Sylvester equations of the form
			\begin{subequations}
			\begin{align}
			A_{\model} \VMr{} - E_\model \VMr{} \SOpt{} - B_\model\ROpt{} &= 0, \label{eq:Proof:Sylv2:V}\\
			\trans{A_\model} \WMr{} - \trans{E_\model} \WMr{} \SOpt{}- \trans{C_\model}\LOpt{} &= 0, \label{eq:Proof:Sylv2:W}
			\end{align} \label{eq:Proof:Sylv2}
			\end{subequations}
			with matrices
			$\SOpt{}\teq\diag\left(\sigma_{\opt,1},\dots,\sigma_{\opt,\ro}\right)$, 
			$\ROpt{}\teq\begin{bmatrix}r_{\opt,1},\dots,r_{\opt,\ro}\end{bmatrix}$,
			$\LOpt{}\teq\begin{bmatrix}l_{\opt,1},\dots,l_{\opt,\ro}\end{bmatrix}$ and
			$\shiftOpti\tin\Complex$, $\tanDirInputOpti\tin\Complex^{\is}$, $\tanDirOutputOpti\tin\Complex^{\os}$.
			Consider the reduced Model Function $\SigmaMr{}$ resulting from the projection  ${\SigmaMr{}\teq\WMrtrans{}\SigmaM{}\VMr{}}$ and let $\GMr(s)$ denote its transfer function.
			 
			Further, assume that for every $i\teq1,\dots,\ro$, the triplets $\tripletOpti$ satisfy $\tripletOpti\teq \left(-\overline{\lambda}_{r,i},\trans{\inRes},\outRes\right)$ where $\GMr(s)\ts=\ts\sum\limits_{i=1}^{\ro} \frac{\outRes\inRes}{s-\lambda_{r,i}}$ is the pole/residue representation of the reduced Model Function.
			
			If, for every $i\teq1,\dots,\ro$, there exists a $j\teq1,\dots,\nM{}$, such that 
			\begin{equation}\label{eq:update condition}
				\tripletMj = \tripletOpti,
			\end{equation}
			then $\SigmaMr{}$ satisfies the first-order \HtwoText\ optimality conditions
			\begin{subequations}
				\begin{align*}
					G(-\bar{\lambda}_{r,i})\trans{\inRes} &=\GMr(-\bar{\lambda}_{r,i})\trans{\inRes}\\
					\trans{\outRes}G(-\bar{\lambda}_{r,i}) &= \trans{\outRes}\GMr(-\bar{\lambda}_{r,i})\\
					\trans{\outRes}G'(-\bar{\lambda}_{r,i})\trans{\inRes} &= \trans{\outRes}\GMrp(-\bar{\lambda}_{r,i})\trans{\inRes}
				\end{align*}
			\end{subequations}
			for all $i=1,\,\dots\,\ro$.
		\end{theorem}
		\begin{proof}
			By \cref{thm:tangential_interpolation} and \cref{thm:Sylvester equivalence}, construction of $\SigmaMr{}$ with the projection matrices satisfying \cref{eq:Proof:Sylv2} results in a reduced model that is a bitangential Hermite interpolant of the Model Function $\SigmaM{}$ with respect to the interpolation data $\SOpt{}$, $\ROpt{}$, and $\LOpt{}$. 
			In combination with the assumption $\tripletOpti\teq\left(-\overline{\lambda}_{r,i},\trans{\inRes},\outRes\right)$, this yields the relationship
			\begin{equation}
				\begin{aligned}
					\GM(-\bar{\lambda}_{r,i})\trans{\inRes} &=\GMr(-\bar{\lambda}_{r,i})\trans{\inRes},\\
					\trans{\outRes}\GM(-\bar{\lambda}_{r,i}) &= \trans{\outRes}\GMr(-\bar{\lambda}_{r,i}),\\
					\trans{\outRes}\GMp(-\bar{\lambda}_{r,i})\trans{\inRes} &= \trans{\outRes}\GMrp(-\bar{\lambda}_{r,i})\trans{\inRes},
				\end{aligned} \label{eq:opt proof1}
			\end{equation}
			for all $i=1,\,\dots\,\ro$.
			
			In addition, from assumption \cref{eq:update condition} and construction of $\SigmaM{}$ through projection with matrices satisfying \cref{eq:Proof:Sylv1} follows
			\begin{equation}
				\begin{aligned}
					G(-\bar{\lambda}_{r,i})\trans{\inRes} &=\GM(-\bar{\lambda}_{r,i})\trans{\inRes},\\
					\trans{\outRes}G(-\bar{\lambda}_{r,i}) &= \trans{\outRes}\GM(-\bar{\lambda}_{r,i}),\\
					\trans{\outRes}G'(-\bar{\lambda}_{r,i})\trans{\inRes} &= \trans{\outRes}\GMp(-\bar{\lambda}_{r,i})\trans{\inRes},
				\end{aligned}\label{eq:opt proof2}
			\end{equation}
			for all $i=1,\,\dots\,\ro$. Equating \cref{eq:opt proof1} and \cref{eq:opt proof2} completes the proof.
		\end{proof}
		All assumptions of \cref{thm:SigmaMr_optimality} are satisfied if $\SigmaMr{}$ results from an \HtwoText-optimal reduction of $\SigmaM{}$ and the Model Function $\SigmaM{}$ is properly updated during \cref{algo:CIRKA}. 
		In fact, the assumption \cref{eq:update condition} can be seen as an \emph{update condition} for the Model Function $\SigmaM{}$. 
		In order to guarantee that the reduced model $\SigmaMr{}$ satisfies the \HtwoText\ optimality conditions with respect to $\Sigma$, it suffices to require $\SigmaM{}$ to be a bitangential Hermite interpolant of $\Sigma$ with respect to the optimal reduction parameters $\SOpt{}$, $\ROpt{}$, $\LOpt{}$.
		
		\cref{thm:SigmaMr_optimality} proves that the new framework results---at convergence---in a reduced model $\SigmaMr{}$ locally satisfying the optimality conditions \cref{eq:OC} and hence effectively solves the \HtwoText-optimal reduction problem \cref{eq:H2 optimization problem}. 
		In addition, it is possible to show that the reduce model $\SigmaMr{}$ has the same state-space realization as one obtained through direct projection of the full model $\Sigma$.

		\begin{corollary}\label{thm:SigmaMr_Sigmar}
			Consider a full-order model $\Sigma$ as in \cref{eq:FOM}.
			Let all assumptions of \cref{thm:SigmaMr_optimality} hold.
			Further, let $V,W\tin\Complex^{\fo\times\ro}$ be projection matrices satisfying Sylvester equations of the form
			\begin{subequations}
			\begin{align}
				A V - E V \SOpt{} - B\ROpt{} &= 0, \label{eq:Proof:Sylv3:V}\\
				\trans{A} W - \trans{E} W \SOpt{}- \trans{C}\LOpt{} &= 0. \label{eq:Proof:Sylv3:W}
			\end{align}\label{eq:Proof:Sylv3}
			\end{subequations}
			Consider the reduced model $\Sigma_r$ resulting from the projection $\Sigma_r\teq\trans{W}\Sigma V$.
			Then it holds 
			\begin{equation}
				\Sigma=\SigmaMr{}.
			\end{equation}		
		\end{corollary}
		\begin{proof}
			The proof amounts to showing that the projection matrices used to obtain $\Sigma_r$ and $\SigmaMr{}$ from projection of $\Sigma$ are equal, hence 
			\begin{subequations}
				\begin{align}
					V&=\VM{}\VMr{}, \label{eq:Proof:VtoVr} \\
					W&=\WM{}\WMr{}. \label{eq:Proof:WtoWr}
				\end{align}
			\end{subequations}
			Consider the Sylvester equation \cref{eq:Proof:Sylv2:V}, for which holds from the definition of $\SigmaM{}$
			\begin{equation}
				\WMtrans{}\left(A\VM{}\VMr{} - E\VM{}\VMr{} \SOpt{} - B\ROpt{}\right) = 0.\label{eq:Proof Corollary: Sylvester M2}
			\end{equation}
			Obviously, by comparing this equation to \cref{eq:Proof:Sylv3:V}, the relation \cref{eq:Proof:VtoVr} is \emph{sufficient} to show that \cref{eq:Proof Corollary: Sylvester M2} is satisfied. 
			In order to show that \cref{eq:Proof:VtoVr} is also \emph{necessary}, we assume that the term in the brackets does not vanish. 
			By defining $\widetilde{V}\defeq\VM{}\VMr{}$, the product \cref{eq:Proof Corollary: Sylvester M2} can be rewritten as
			\begin{equation}
				\begin{bmatrix}
					\vdots \\
					\trans{\tanDirOutputMj}C\inv{\left(A-\shiftMj E\right)}\\
					\vdots
				\end{bmatrix}
				\cdot
				\left(A\widetilde{V}- E\widetilde{V} \begin{bmatrix} \ddots &&\\ & \shiftOpti &\\ &&\ddots\end{bmatrix} - B\begin{bmatrix}
					\cdots & \tanDirInputOpti &\cdots \end{bmatrix}\right) = 0, \label{eq:Proof Corollary: Sylvester M2b}
			\end{equation}
			which needs to hold true for all $j=1,\dots,\nM{}$ and $i=1,\dots,\ro$.
			Due to the update condition \eqref{eq:update condition}, for every $i$ there exists a $j$ such that $\tripletMj\teq\tripletOpti$. For all such $i,j$ combinations, we can hence consider the inner product between the $j\textsuperscript{th}$ row of $\WMtrans{}$ and the $i\textsuperscript{th}$ column in the bracket
			\begin{align}
					\trans{\tanDirOutputOpti}C\inv{\left(A-\shiftOpti E\right)}\left(\left(A -\shiftOpti E\right)\widetilde{V}\ei - B\tanDirInputOpti\right) &= 0, \nonumber \\
					\iff \trans{\tanDirOutputOpti}C\left(\widetilde{V}\ei - \inv{\left(A-\shiftOpti E\right)}B\tanDirInputOpti\right) &= 0, \nonumber\\
					\iff \trans{\tanDirOutputOpti}C\left(\widetilde{V}\ei - \inv{\left(A-\shiftOpti E\right)}B\tanDirInputOpti\right) &= 0,\label{eq:Prof Corollary: SylvesterM3}
			\end{align}
			where $e_i\in\Reals^{\ro}$ is a vector with 1 on the i\textsuperscript{th} entry and otherwise 0.
			Since \cref{eq:Prof Corollary: SylvesterM3} needs to hold true for all possible $C$, it finally follows that
			\begin{equation}
				\widetilde{V}\ei = \inv{\left(A-\shiftOpti E\right)}B\tanDirInputOpti,\label{eq:Proof Corollary: Sylvester M4}
			\end{equation}
			which by \cref{thm:Sylvester equivalence} and \cref{eq:Proof:Sylv3:V} corresponds exactly to $V\ei$.
			As \cref{eq:Proof Corollary: Sylvester M4} holds true for all columns $i\teq1,\dots,\ro$, we finally obtain $\widetilde{V}=\VM{}\VMr{}=V$. The proof for $W$ is analogous.
		\end{proof}
		The results of \cref{thm:SigmaMr_Sigmar} require all projection matrices to be primitive bases (cp. \cref{def:primitive basis}). In practical applications, this is not the case. In general, the realizations $\SigmaMr{}$ and $\Sigma_r$ will be \emph{restricted system equivalent}, i.e. sharing the same order and transfer function \cite{morMayA07}.

	\subsection{Interpretations of the New Framework}
		The new framework introduced in this section, which we refer to as the \emph{Model Function} framework in general or Confined IRKA (CIRKA) when applied to IRKA, can be given different interpretations.
		
		On the one hand, it is a form of \emph{surrogate optimization} \cite{queipo2005surrogate,forrester2008engineering} in that the \HtwoText\ optimization is not conducted on the actual cost function $\mathcal{J}\ts=\ts\norm{G-G_r}{\Htwo}$ but on an approximation ${\mathcal{J}\ts\approx\ts\widehat{\mathcal{J}}\ts=\ts \norm{\GM-G_r}{\Htwo}}$.
		In fact, the framework itself is an application of \emph{reduced-model based optimization} \cite{morBenTT13,legresley2000airfoil,han2005efficient}. 
		On the other hand, the framework can be seen as \HtwoText\ optimization in a subspace, defined by the matrices $\VM{}$ and $\WM{}$, that is updated at every iteration. In fact, this technique could be interpreted as a \emph{subspace acceleration} method \cite{Rommes_2006,sirkovic2016subspace,washio1997krylov} to recycle information obtained in a previous optimization step. 
		Traditionally, subspace acceleration methods in numerics are used to ameliorate the convergence of iterative methods. In our setting, this becomes the convergence to a set of optimal reduction parameters.
		Finally, one may think of it as a \emph{restarted} \HtwoText\ optimization, where e.g. IRKA is restarted in a higher-dimensional subspace after convergence.

	\subsection{The Cost of the New Framework}\label{sec:costOfNewFramework}
		The discussion so far has shown how the reduced order model $\SigmaMr{}$, resulting from the Model Function framework, indeed satisfies \HtwoText-optimality conditions. What still remains to be discussed is whether we can expect this framework to be computationally less demanding than simply applying \HtwoText-optimal reduction to the full order model.
		
		To address this question, we compare the cost of IRKA, estimated in \cref{eq:Cost of IRKA}, to the cost of CIRKA.
		Following similar considerations as in \cref{sec:costOfH2}, we obtain the relationship
		\begin{equation}
		     \Cost{\text{CIRKA}}{\fo}\approx  \underbrace{\sum\limits_{k=1}^{\kCirka}\nM{k,+}\cdot\Cost{\text{LSE}}{\fo}}_{\text{cost of reduction}} + \underbrace{\sum\limits_{k=1}^{\kCirka}2\ro\cdot\kIrka^k\cdot\Cost{\text{LSE}}{\nM{k}}}_{\text{cost of optimization}}
			\label{eq:Cost of CIRKA}
		\end{equation}	
		where $\kCirka$ is the number of iterations of the new framework, $\nM{k,+}$ indicates the size of the Model Function update in each step, and $\kIrka^k$ is the number of IRKA iterations in each step of the new framework.
		The first summand represents the cost of updating the Model Function with information of the full-order model and could be interpreted as the \emph{cost of reduction}.
		The second summand adds up the cost of IRKA in each iteration and depends on the number of optimization steps in IRKA $\kIrka^k$, as well as $\nM{k}$-dimensional LSEs. 
		As it can be seen from \cref{eq:Cost of CIRKA}, the cost involved in finding optimal parameters is in some sense \emph{decoupled} from the cost of reduction, although a link through $\kCirka$ still remains.
		Obviously, as $\fo$ increases, we expect $\Cost{\text{LSE}}{\fo}\ts\gg\ts\Cost{\text{LSE}}{\nM{i}}$. 
		Hence the dominant cost is represented by the cost of reduction, whereas the number of optimization steps $\kIrka^k$ plays no role.
		
		Therefore, comparing \cref{eq:Cost of IRKA} and \cref{eq:Cost of CIRKA} we expect the new framework to be cost-effective as long as
		\begin{equation}
			  \sum\limits_{k=1}^{\kCirka}\nM{k,+} < 2\ro\cdot\kIrka,
			\label{eq:Cost comparison}
		\end{equation}
		where $\kIrka$ in this case represents the number of iterations required by IRKA to directly reduce the full order model.
		In \cref{sec:numericalresults}, numerical results will illustrate the substantial speedup that can be achieved by using this framework.
		
	\subsection{Further Advantages of the New Framework}\label{sec:advantagesOfNewFramework}
		We conclude this section by addressing a major advantage that this framework bares, compared to conventional \HtwoText\ optimization, in addition to the speedup already discussed.
		In fact, in the process of generating an \HtwoText-optimal reduced-order model $\Sigma_r$, this framework naturally yields an \emph{additional} reduced-order model $\SigmaM{}$ at no additional cost. 
		This Model Function has a reduced order $\nM{}\ts>\ts\ro$ and is---though not being \HtwoText-optimal---in general a better approximation to $\Sigma$ than $\SigmaMr{}$. This is especially true if update strategies U.1 or U.2 are used.
		For this reason, it is possible to use this free information about the full model, e.g. to estimate the relative reduction error $\error$ by 
		\begin{equation}
			\error = \frac{\norm{G-G_r}{\Htwo}}{\norm{G}{\Htwo}} \approx \frac{\norm{\GM-G_r}{\Htwo}}{\norm{\GM}{\Htwo}} =: \errorEst. \label{eq:errorEst}
		\end{equation}
		As error estimation in model reduction by tangential interpolation is still a big open challenge, this type of indicator is certainly of great interest, even though no claim about its rigorosity can be made.
		Note that this error estimation is very similar to the one in \cite{morCasBG17}, where data-driven approaches are used to generate a surrogate model from data acquired during IRKA.
		As reduced models obtained by projection can become unstable, in general it is necessary to take the stable part of $\SigmaM{}$ after convergence to be able to evaluate \cref{eq:errorEst}. As $\nM{}$ is small, This can be easily computed, for example in MATLAB using the \mcode{stabsep} command. 
		Note that even though in theory no claim can be made about the stability of reduced order models obtained by IRKA, in practice more often then not the resulting reduced models are stable \cite{}. We make the same observation for CIRKA.
		
		In addition, by making the cost of optimization negligible with respect to the cost of reduction, the new framework allows to introduce \emph{globalized} local optimization approaches \cite{pinter2013global}, where the local reduction is performed from different start points, possibly finding several local minima and increasing the chance of finding the global minimum. 
		Global \HtwoText\ optimal reduction based on the Model Function framework is a topic of current research and will be presented in a separated paper.
		
		

\section{Numerical Results}\label{sec:numericalresults}

\subsection{CD Player Model}
	We compare CIRKA to IRKA on the model \mcode{CDplayer} taken from the benchmark collection \cite{morChaV02}.
	The model has a low full order $\fo\teq120$ and can therefore be used for illustrative examples. 
	It has $\is\teq\os\teq2$ inputs and outputs. 
	\HtwoText-optimal reduced models for the reduced orders ${\ro\teq10, 20, 30}$ are constructed using both IRKA and CIRKA,  
	run using the \mcode{sssMOR} toolbox implementations. The execution parameters were set to default, in particular the chosen convergence tolerance was set to \mcode{Opts.tol = 1e-3} and the convergence was determined by inspection of the optimal frequencies $\shiftOpti$ (\mcode{Opts.stopCrit = 's0'}) in IRKA and optimal frequencies $\shiftOpti$ and tangential directions $\tanDirInputOpti$, $\tanDirOutputOpti$ (\mcode{Opts.stopCrit = 's0+tanDir'}) in CIRKA. 
	The initial frequencies in $S^0$ were all set to $0$, whereas all initial tangential directions in $R^0$ and $L^0$ were set to $\trans{\begin{bmatrix}1 & 1	\end{bmatrix}}$.
	In CIRKA, strategies I.2 for initialization and U.2 for Model Function updates were used.

	\begin{table}[h!]
	\centering
	\caption{Reduction results for the \mcode{CDplayer} model using zero initialization.}
	\label{tab:CDplayer:zero}
	\begin{tabular}{c|ccc|ccccc}
		\toprule
		& \multicolumn{3}{c}{IRKA} & \multicolumn{5}{c}{CIRKA} \\
		\midrule
		$\ro$ 	& $\kIrka$ 	& $\nLU$	& $\error$ 	& $\kCirka$ 	& $\sum\kIrka$ 	& $n_{LU}$ 	& $\error$ 	& $\errorEst$ \\
		\midrule
		10 		& 11 		& \textcolor{red}{52} 		& 5.92e-5 	& 3 			& 75 			& \textcolor{green}{10} 		& 5.92e-5  	& 5.72e-5 \\
		20 		& 50 		& \textcolor{red}{493} 		& 1.64e-5 	& 3 			& 150 			& \textcolor{green}{14} 		& 1.64e-5  	& 8.16e-6 \\
		30 		& 29 		& \textcolor{red}{438} 		& 1.76e-6 	& 4 			& 76 			& \textcolor{green}{25} 		& 1.76e-6  	& 1.77e-6 \\
		\bottomrule			
	\end{tabular}
\end{table}

	%
	The results are summarized in \cref{tab:CDplayer:zero}.
	For the case $\ro\teq10$, IRKA converged within $\kIrka\teq11$ steps, requiring $\nLU\teq52$ full-dimensional LU decompositions. On the other hand, CIRKA converged after $\kCirka\teq3$ steps (i.e. two updates of $\SigmaM{}$). Even though in sum it required 75 optimization steps to find the optimal frequencies (as opposed to 11 for IRKA), the number of full-dimensional LU decompositions was only $\nLU\teq10$, i.e. approximately five times less than IRKA. 
	In addition, using the Model Function $\SigmaM{}$, we are able to estimate the reduction error accurately to the first significant digit.
	For the reduced orders $\ro\teq20$ and  $\ro\teq30$, IRKA required more steps for convergence, resulting in a much higher number of full-dimensional LU decompositions. In fact, for the case $\ro\teq20$, the maximum iteration tolerance of \mcode{Opt.maxiter=50} was reached without convergence. Note that in this case, no claim about the optimality of the reduced model can be made, even though in practice IRKA models tend to yield good approximations even without convergence.
	Also CIRKA required more optimization steps for the case $\ro\teq20$, however the optimal parameters converged within $\kCirka\teq3$ steps, thus limiting the number of full dimensional LU decomposition to $14$, i.e. over 30 times less than IRKA.
	
	\setlength{\mywidth}{12cm}
	\setlength{\myheight}{6cm}
	\begin{figure}[h!]
		\centering
		\input{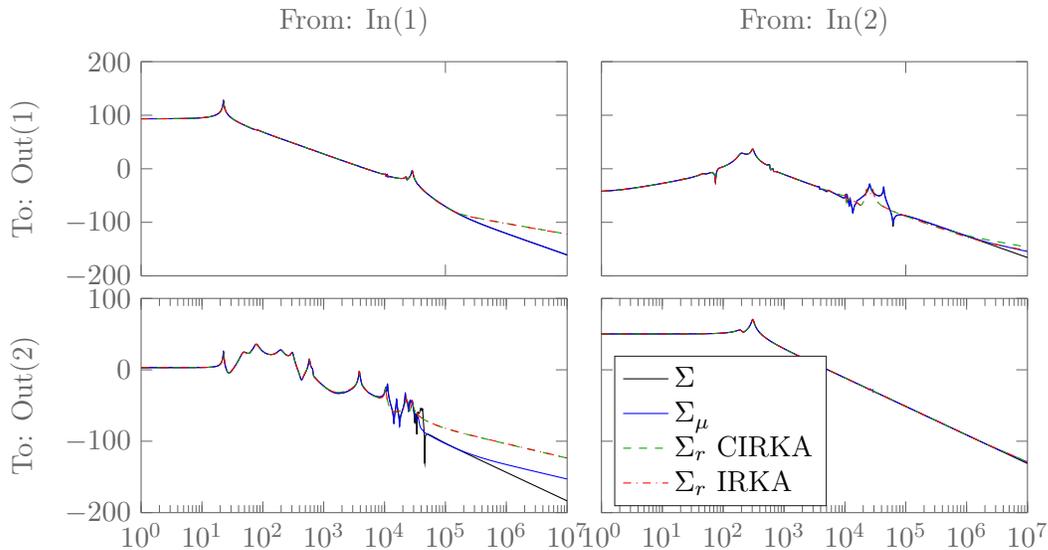}
		\caption{Reduction of the CDplayer model using CIRKA and IRKA ($\ro=30$).}
		\label{fig:CDplayer}
	\end{figure}
	\cref{fig:CDplayer} shows the bode plots of full and reduced models for the case $\ro\teq30$.
	As it can be seen, the reduced models obtained from IRKA and CIRKA are the same. This is also reflected in \cref{tab:CDplayer:zero}, where the relative reduction errors $\error$ are the same for both algorithms. 
	In addition, the blue line shows the Model Function $\SigmaM{}$, which is a more accurate approximation to the full model $\Sigma$, motivating its use for error estimation.
	
	To analyze the dependency of the results in \cref{tab:CDplayer:zero} from the initialization, \cref{tab:CDplayer:eigs} summarizes the results for initial frequencies and tangential directions corresponding to the mirrored eigenvalues with smallest magnitude and respective input/output residues of the full model $\Sigma$.
	This choice of initial parameters is known to be advantageous \cite{Gugercin_PhD,Gugercin_2003_Error,morGugAB08,morPan14}, its computation can be performed efficiently for large-scale sparse models using iterative methods (cp. e.g. the MATLAB command \mcode{eigs}).
	As it can be seen from \cref{tab:CDplayer:eigs}, the results are similar to the ones obtained through the previous initialization. 
	In particular, it is worth noticing that for the case $\ro\teq20$, also CIRKA suffers from the convergence issues within IRKA, requiring $\kCirka\teq9$ iterations and $\sum\kIrka\teq450$ optimization steps to converge. Nonetheless, as only a few frequencies and tangential directions exhibit slow convergence, the update cost is minimal, resulting in a total of $\nLU\teq30$ full-dimensional LU decompositions required to find an optimum. This indicates that CIRKA can in some sense mitigate the effect of slow convergence of IRKA.
	\begin{table}[h!]
	\centering
	\caption{Reduction results for the \mcode{CDplayer} model using \mcode{eigs} initialization.}
	\label{tab:CDplayer:eigs}
	\begin{tabular}{c|ccc|ccccc}
		\toprule
		& \multicolumn{3}{c}{IRKA} & \multicolumn{5}{c}{CIRKA} \\
		\midrule
		$\ro$ 	& $\kIrka$ 	& $\nLU$					& $\error$ 	& $\kCirka$ 	& $\sum\kIrka$ 	& $n_{LU}$ 						& $\error$ 	& $\errorEst$ \\
		\midrule
		10 		& 7 		& \textcolor{red}{36} 		& 5.92e-5 	& 3 			& 69 			& \textcolor{green}{14} 		& 5.92e-5  	& 5.72e-5 \\
		20 		& 50 		& \textcolor{red}{501} 		& 1.64e-5 	& 9 			& 450 			& \textcolor{green}{30} 		& 1.64e-5  	& 1.24e-5 \\
		30 		& 27 		& \textcolor{red}{420} 		& 2.11e-6 	& 4 			& 65 			& \textcolor{green}{29} 		& 2.11e-6  	& 2.16e-6 \\
		\bottomrule			
	\end{tabular}
\end{table}

	%

\subsection{Butterfly Gyroscope Model}
	In this section, we compare IRKA and CIRKA in the reduction of a model of larger size, taken from the benchmark collection \cite{morKorR05}, and include a comparison of reduction times.
 	The model (\mcode{gyro}) represents a micro-electro mechanical Butterfly gyroscope \cite{Lienemann_2004} and has a full order of $\fo\tapprox 35'000$, $\is\teq1$ input and $\os\teq12$ outputs.
	The computations were conducted on an Intel\textsuperscript{\textregistered} Xeon\textsuperscript{\textregistered} CPU @ \unit[2.27]{GHz} computer with four cores and \unit[48]{GB} RAM. 
	The results for zero initialization are summarized in \cref{tab:gyro:zero}.
	
	\begin{table}[h!]
	\centering
	\caption{Reduction results for the \mcode{gyro} model using zero initialization.}
	\label{tab:gyro:zero}
	\begin{tabular}{c|cccc|cccccc}
		\toprule
		& \multicolumn{4}{c}{IRKA} & \multicolumn{6}{c}{CIRKA} \\
		\midrule
		$\ro$ 	& $\kIrka$ 	& $\nLU$	& $t /s$ & $\error$ 	& $\kCirka$ & $\sum\kIrka$ 	& $n_{LU}$ & $t /s$	& $\error$ 	& $\errorEst$ \\
		\midrule
		10 	& 50 & \textcolor{red}{246} & \textcolor{red}{887} & 0.45 	& 4 & 110 	& \textcolor{green}{8} 	& \textcolor{green}{40} & 0.21	& 0.26 \\
		20 	& 7  & \textcolor{red}{61}  & \textcolor{red}{226} & 0.12 	& 2 & 19 	& \textcolor{green}{11} & \textcolor{green}{51} & 0.12  & 0.04 \\
		30 	& 24 & \textcolor{red}{356} & \textcolor{red}{1281}& 0.05   & 2 & 37 	& \textcolor{green}{17} & \textcolor{green}{83} & 0.05	& 3.32e-3 \\
		\bottomrule			
	\end{tabular}
\end{table}

	
	For all three cases considered, CIRKA required significantly less full-dimensiona LU decompositions compared to IRKA. This is reflected also in the significant speedup in reduction time, ranging from four to over 20 times faster than IRKA. 
	Note that for the case $\ro\teq10$, IRKA did not converge within the maximum number of steps. On the other hand, CIRKA converges to a local optimum, hence reaching a lower relative approximation error.
	As the model represents an oscillatory system with weakly damped eigenfrequencies and many outputs, it is harder to approximate well with a low reduced order, as the relative approximation errors in \cref{tab:gyro:zero} reflect. Nevertheless, a comparison of the reduced models in a Bode diagram shows a good approximation, as it is depicted  in \cref{fig:gyro} exemplary for the case $\ro\teq30$ and the 9\textsuperscript{th} output.
	\setlength{\mywidth}{12cm}
	\setlength{\myheight}{4cm}
	\begin{figure}[h!]
		\centering
		\input{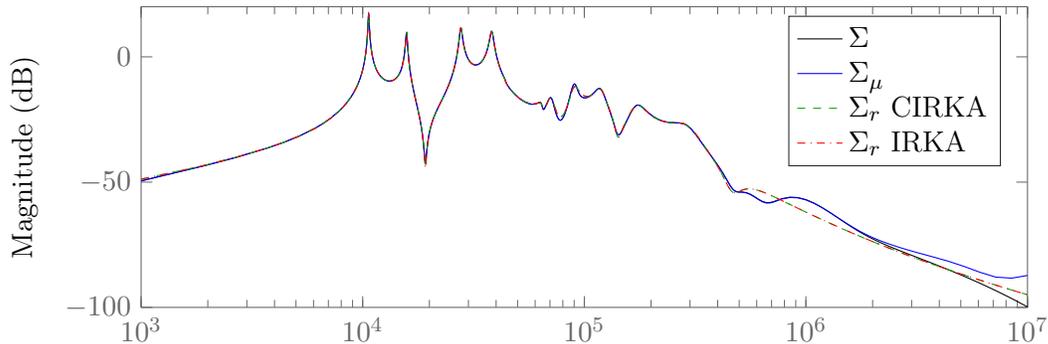}
		\caption{Reduction of the \mcode{gyro} model using CIRKA and IRKA ($\ro\teq30$, 9\textsuperscript{th} output).}
		\label{fig:gyro}
	\end{figure}
	
	The same reduction was conducted initializing frequencies and tangential directions based on the eigenvalues with smallest magnitude and corresponding residual directions. The results are shown in \cref{tab:gyro:eigs}.
	
	\begin{table}[h!]
	\centering
	\caption{Reduction results for the \mcode{gyro} model using \mcode{eigs} initialization.}
	\label{tab:gyro:eigs}
	\begin{tabular}{c|cccc|cccccc}
		\toprule
		& \multicolumn{4}{c}{IRKA} & \multicolumn{6}{c}{CIRKA} \\
		\midrule
		$\ro$ 	& $\kIrka$ 	& $\nLU$	& $t /s$ & $\error$ 	& $\kCirka$ & $\sum\kIrka$ 	& $n_{LU}$ & $t /s$	& $\error$ 	& $\errorEst$ \\
		\midrule
		10 	& 1  & \textcolor{green}{5} & \textcolor{green}{19}& 0.28 	& 2 & 8 	& \textcolor{red}{6} 	& \textcolor{red}{25}   & 0.41	& 0.26 \\
		20 	& 9  & \textcolor{red}{90}  & \textcolor{red}{323} & 0.17 	& 2 & 30 	& \textcolor{green}{14} & \textcolor{green}{60} & 0.12  & 0.05 \\
		30 	& 9  & \textcolor{red}{136} & \textcolor{red}{493} & 0.03   & 2 & 40 	& \textcolor{green}{22} & \textcolor{green}{96} & 0.02	& 5.57e-3 \\
		\bottomrule			
	\end{tabular}
\end{table}

	
	Interestingly, for the case $\ro\teq10$, IRKA converged to the default tolerance in one step, requiring only 5 LU decompositions. As CIRKA required two iterations until convergence, the number of full-dimensional LU decomposition required were 6. In this special case, IRKA was faster than CIRKA. In addition, comparison of the relative \HtwoText\ error indicates that IRKA converged to a better local minimum.
	A comparison with the respective result of \cref{tab:gyro:zero} demonstrates that initialization of \HtwoText-optimal reduction can have a large impact on the optimization and the results.
	For the other two cases, the pattern is similar to what discussed so far. In particular, the speedup obtained through CIRKA is significant.
	
	Finally, note that even when IRKA converges in few steps (e.g. 7 or 9) a reduction using CIRKA can already result in a significant speedup. Also note that in general, the relative error estimation $\errorEst$ tends to \emph{under}estimate the approximation error, especially for models with high dynamics.
	



\section{Application to Different System Classes}\label{sec:furtherApplications}
	The discussion of this contribution has been limited to the model reduction of linear time-invariant systems as in \eqref{eq:FOM}, described by a set of ordinary differential equations. 
	However, \HtwoText\ approximation approaches exist for other system classes as well, including systems of Differential Algebraic Equations (DAEs) \cite{morGugSW13}, systems for which only transfer function evaluations or measurements are available \cite{morBeaG12}, linear systems with time delays \cite{DUFF20167}, linear systems in port-Hamilton form \cite{GUGERCIN20121963}, as well as nonlinear systems in bilinear \cite{morBenB12b,morFlaG15} or quadratic-bilinear form \cite{benner2016mathcal}.
	
	Obviously, going into the details of all approximation methods above would exceed the scope of this paper. 
	Nonetheless, we feel that also \HtwoText\ algorithms for different system classes could greatly benefit from an approach similar to the one presented in this contribution. 
	Generally speaking, the idea of surrogate optimization is certainly not new.
	However, as most methods above are based on \emph{interpolation}, an \emph{update} of the surrogate model as presented here can ensure interpolation of the full model at the optimal frequencies, which would be otherwise lost.
	Once more, we emphasize that the idea presented in this contribution truly is a \emph{general framework} and not just an additional reduction algorithm by the name of CIRKA.
	In the following, we briefly indicate for some system classes, how the principle of this new framework could be applied. The goal is to reduce the computational cost while still satisfying the same conditions as the original algorithms.
	
	\subsection{DAE Systems}
		DAE systems are state-space representations as in \cref{eq:FOM} having a singular matrix $E$. 
		Their transfer function can be generally decomposed into the sum of a strictly proper part $G_{sp}(s)$, satisfying $\lim\limits_{s\to\infty}G_{sp}(s)\teq0$, and a polynomial part $P(s)$, a polynomial in $s$ of order at most $\nu$, the so called \emph{index} of the DAE \cite{morBenS15}.
		Due to this polynomial contribution, approximation by tangential interpolation is not sufficient to prevent the reduction error to become unbounded, which is a result of a mismatch in the polynomial part. For this reason, in the reduction the polynomial part needs to be matched exactly, while the strictly proper part can be approximated through tangential interpolation.
		To achieve tangential interpolation of the strictly proper part while preserving the polynomial part, the subspace conditions in \cref{eq:input_Krylov}, \cref{eq:output_Krylov} can be adapted by including \emph{spectral projectors} onto \emph{deflating subspaces} \cite[Theorem 3.1]{morGugSW13}. 
		Based on this, Gugercin, Stykel and Wyatt show in \cite[Theorem 4.1]{morGugSW13} that interpolatory \HtwoText-optimal reduction of DAEs can be performed by using the modified subspace conditions to find an \HtwoText\ approximation to the strictly proper part $G_{sp}(s)$.
		
		For this system class, an extension of the Model Function framework is quite straightforward: 
		By computing a Model Function following the modified interpolatory subspace conditions of \cite[Theorem 3.1]{morGugSW13}, \HtwoText\ optimization can be performed on the surrogate instead of on the original DAE. An update of the Model Function finally leads to a reduced order model that tangentially interpolates the strictly proper part of the original DAE at the optimal frequencies, along the optimal right and left tangential directions.

	\subsection{Port-Hamiltonian Systems}
		Port-Hamiltonian (pH) formulations of linear dynamical systems are a very effective way to model systems that result from the interconnection of different subsystems, especially in a multi-physics domain. 
		Their representation is similar to \cref{eq:FOM}, where the system matrices have the special structures $E\teq I$, $A\teq(J-R)Q$, and $C\teq\trans{B}Q$, $J$ being a skew-symmetric interconnection matrix, $R$ a positive semi-definite dissipation matrix and $Q$ a positive definite energy matrix \cite{van2000l2}.
		Models in pH form bear several advantages, such as passivity, and it is therefore of great relevance to preserve the pH structure after reduction. In a setting of tangential interpolation, this can be achieved by computing a projection matrix $V$ according to \cref{eq:input_Krylov} and choosing $W$ in order to retain the pH structure (cp. \cite{Wolf_2010,GUGERCIN20121963}).
		
		Gugercin, Polyuga, Beattie and van der Schaft present in \cite{GUGERCIN20121963} an iterative algorithm by the name of IRKA-PH, that adaptively chooses interpolation frequencies and input tangential directions, inspired by the IRKA iteration. In general, the resulting reduced model will not satisfy the first-order optimality conditions \cref{eq:OC}, as $W$ is not chosen to enforce bitangential Hermite interpolation but rather structure preservation. 
		Nonetheless, the algorithm has shown to produce good approximations in practice.
		
		To reduce the cost involved in repeatedly computing the projection matrices \cref{eq:input_Krylov} in IRKA-PH, a Model Function in pH form could be introduced and updated. 
		At convergence, the reduced model satisfies the condition \cref{eq:OC1} while being of pH-form as in the IRKA-PH case. In addition, the resulting Model Function (also in pH form) could be used e.g. for error estimation.
		
	\subsection{TF-IRKA}
		Gugercin and Beattie demonstrate in \cite{morBeaG12} how to construct a reduced order model $G_r(s)$ satisfying first-order \HtwoText\ optimality conditions \cref{eq:OC} only from evaluations of the transfer function $G(s)$. 
		The algorithm by the name of TF-IRKA exploits the \emph{Loewner-Framework}, a data-driven approximation approach by Mayo and Antoulas \cite{morMayA07} that produces a reduced model that tangentially interpolates the transfer function of a given system, whose transfer behavior can be measured or evaluated at selected frequencies. 
		Therefore, this algorithm can be used to obtain a reduced order model satisfying optimality conditions for irrational, infinite dimensional dynamical systems for which an expression for $G(s)$ can be obtained by direct Laplace transformation of the partial differential equation, i.e. without prior discretization. In addition, it can be used for optimal approximation of systems where no model is available but the transfer behavior $G(s)$ can be measured in experiments. Similar approaches have been derived to obtain reduced order models for systems with delays \cite{morDufPVS15}.
		
		The cost of TF-IRKA is dominated by the evaluation of the transfer function $G(s)$. In general, if an analytic expression is given, this cost is minimal compared to evaluating a discretized high-order model. However, if the evaluation is obtained through costly measurements, an approach to reduce the number of evaluations would be highly beneficial. 
		In \cite{morBeaDG15}, Beattie, Drma\v{c}, and Gugercin introduce a quadrature-based version of TF-IRKA for SISO models, called Q-IRKA, that only requires the evaluation of $G(s)$ once for some frequencies (quadrature nodes) and returns a reduced order model that satisfies the optimality conditions \cref{eq:OC} within the quadrature error. 
		In \cite{morDrmGB15}, the same authors propose a similar approach by the name of QuadVF, based on \emph{vector fitting} rational approximations using  frequency sampling points
		
		By exploiting the new framework presented in this paper, it is possible to reduce the number of evaluations of $G(s)$ \emph{and} satisfy the optimality conditions \cref{eq:OC} exactly. By creating a surrogate using the Loewner-Framework, standard IRKA can be applied to obtain a new set of frequencies and tangential directions. Evaluating (or measuring) $G(s)$ for the new optimal frequencies allows to update the surrogate and repeat the process. Compared to Q-IRKA, this approach bears the additional advantage of automatically finding the amount and position of suitable frequencies for which the evaluation of $G(s)$. In contrast, Q-IRKA and QuadVF require the initial evaluation of $G(s)$ at some pre-specified quadrature nodes. While a judicious spacing of the sampling frequencies in QuadVF was presented in \cite{morDrmGB15}, the frequency range and number of sampling points remains a tunable parameter within this approaches.
	
	\subsection{Nonlinear Systems}
		Model reduction approaches for general nonlinear systems typically differ from what discussed in this contribution. 
		However, in the case of weak nonlinearities, a bilinear formulation can be obtained by means of Carleman Bilinearization \cite{rugh1981nonlinear}. 
		For this class of nonlinear systems, the dynamic equations are similar to \cref{eq:FOM} and are complemented by an additional term that sums the weighted product between the state and the inputs.
		Recently, Benner and Breiten have presented in \cite{morBenB12b} a Bilinear IRKA (B-IRKA) algorithm that produces a reduced bilinear model satisfying, at convergence, first-order optimality conditions with respect to the bilinear \HtwoText\ norm for the approximation error. This algorithm requires the repeated solution of two bilinear Sylvester equations, which can be obtained either by vectorization, solving linear systems of the dimension $\fo\,\ro$, or, equivalently, computing the limit to an infinite series of linear Sylvester equations.
		Flagg and Gugercin have demonstrated in \cite{morFla12,morFlaG15} for SISO models that the optimality conditions in \cite{morBenB12b} can be interpreted as \emph{interpolatory} conditions of the underlying Volterra series (\cite[Theorem 4.2]{morFlaG15}). In addition, they present an algorithm to construct reduced bilinear models that achieve multipoint Volterra series interpolation by construction (cp. \cite[Theorem 3.1]{morFlaG15}), requiring the solution of two bilinear Sylvester equations.
		Exploiting the often fast convergence of the Volterra integrals, they propose a Truncated B-IRKA (TB-IRKA) that satisfies the first-order optimality conditions asymptotically as the truncation index tends to infinity. 
		Note that extensions to bilinear DAEs were presented by Benner and Goyal in \cite{morBenG16}.
		
		The interpolatory nature of this \HtwoText\ approach makes it possible to extend the new framework presented in this paper to bilinear systems, hopefully speeding up the reduction with B-IRKA significantly.
		In fact, given a set of initial frequencies, a bilinear Model Function can be constructed that interpolates the Volterra series at this frequencies. This would require the solution of the full bilinear Sylvester equations for a few selected frequencies. Using this bilinear surrogate model, B-IRKA can be run efficiently on a lower dimension to find a new set of optimal frequencies and update the Model Function. Note that the cost of solving a full-dimensional bilinear Sylvester equation plays the same role in the complexity of B-IRKA as the cost of a full-dimensional LU decomposition for linear systems. Assuming the convergence behavior of the new framework is similar to the linear case, this could lead to an algorithm that produces reduced bilinear models that satisfy the \HtwoText\ optimality conditions \emph{exactly}---i.e. without truncation error---at a far lower cost.
		
		In presence of strong nonlinearities, the dynamics can be captured \emph{exactly} in a quadratic bilinear model where, in addition, a quadratic term in the state variable is present. An \HtwoText\ approach for this system class has been recently presented by Benner, Goyal and Gugercin in \cite{benner2016mathcal} and it is not evident at this point, how the new framework could be applied to this system class while preserving rigor of the results.


\section{Conclusions and Outlook}\label{sec:conclusions}
	In this paper we have presented a new framework for the \HtwoText-optimal reduction of MIMO linear systems. This new framework is based on the local nature of tangential interpolation and \HtwoText\ optimality. 
	By means of \emph{updated} surrogate optimization, the cost of reduction can be decoupled from the cost of optimization.
	Several numerical examples have illustrated the effectiveness of the new framework in reducing the cost of \HtwoText-optimal reduction. 
	Theoretical considerations have demonstrated how the reduced order models resulting from this new framework still satisfy first-order optimality conditions at convergence. 
	First indications have been given on how to extend this new framework for \HtwoText\ reduction of different system classes, e.g. DAEs and bilinear systems.
	The new framework not only produces optimal reduced models at a far lower cost than conventional methods, but also provides---at not additional cost---a middle-sized surrogate model that can be used e.g. for error estimation.
	
	Current research endeavors exploit the new framework to obtain the \emph{global} optimum amongst all local \HtwoText\ optimal reduced models of prescribed order \cite{morCasHL18}. In addition, the Model Function can be used for error estimation in the cumulative reduction framework CURE by Panzer, Wolf and Lohmann \cite{morPanJWetal13,morPan14,morWol15} to adaptively chose the reduced order. Finally, the Model Function framework is currently being used in parametric model reduction to recycle the Model Function from one parameter sample point to another and reduced the cost tied to the repeated reduction at several points in the parameter space.
	These advances will be topics of future publications.

\printbibliography
\end{document}